\documentclass[11pt]{amsart}

\usepackage[utf8]{inputenc}
\usepackage[text={6.1in,8.5in},centering]{geometry}
\usepackage{amssymb,amsmath,amsthm,bm,fdsymbol}
\usepackage{tikz}
\usepackage[square,numbers]{natbib}
\usepackage{placeins}
\usepackage{hyperref}
\usepackage{multirow}
\usepackage[table]{xcolor}
\usepackage{makecell}
\usepackage{xcolor}

\usepackage[table]{xcolor}
\usepackage{makecell}

\definecolor{cryst}{gray}{0.82}  
\definecolor{precr}{gray}{0.92}  

\newcommand{\crystcell}[1]{\cellcolor{cryst}\textbf{#1}}
\newcommand{\precrcell}[1]{\cellcolor{precr}{#1}}

\usetikzlibrary{patterns}
\usetikzlibrary{fadings}

\newtheorem{thm}{Theorem}[section]
\newtheorem*{thm*}{Theorem}
\newtheorem*{lem*}{Lemma}

\newtheorem{lem}[thm]{Lemma}
\newtheorem{prop}[thm]{Proposition}
\newtheorem*{prop*}{Proposition}
\newtheorem{cor}[thm]{Corollary}
\newtheorem*{conj}{Conjecture}

\theoremstyle{definition}
\newtheorem{defn}[thm]{Definition}
\newtheorem{rmk}[thm]{Remark}

\numberwithin{equation}{section}

\newcommand{\Pmin}{\mathcal{P}_{min}}      
\newcommand{\vmin}{\mathcal{V}_{min}}      

\newcommand{\polyscale}{.155\linewidth}

\title{Holey Hyperbolic Polyforms}

\author[S. Eldridge]{Summer Eldridge}
\address[S. Eldridge]{Department of Mathematics, Graduate Center, City University of New York.}
\email{seldridge@gradcenter.cuny.edu}

\author[A. Prabha]{Adithya Prabha}
\address[A. Prabha]{Basis Independent McLean, Virginia.}
\email{prabhaadithyasai@gmail.com}

\author[A. Roger]{Aiden Roger}
\address[A. Roger]{Department of Mathematics, George Mason University.}
\email{aroger29@gmu.edu}

\author[C. Roger]{Cooper Roger}
\address[C. Roger]{Department of Mathematics, George Mason University.}
\email{croger20@gmu.edu}

\author[E.~Rold\'an]{\'Erika Rold\'an}
\address[E.~Rold\'an]{Max Planck Institute for Mathematics in the Sciences, Leipzig; Department of Statistical Learning, ScaDS.AI, Leipzig University; and Universit\'e de Gen\`eve.}
\email{roldan@mis.mpg.de}

\author[R.~Toal\'a]{Rosemberg Toal\'a-Enr\'iquez}
\address[R.~Toal\'a-Enr\'iquez]{Departamento de Matemáticas y Física, Universidad Autónoma de Aguascalientes.}
\email{rosemberg.toala@edu.uaa.mx}

\begin{document}
\begin{abstract}
A polyform is a planar figure formed by gluing congruent regular polygons along entire edges.  
We study polyforms in hyperbolic $\{p,q\}$-tessellations and the extremal problem of minimizing the number of tiles needed to realize exactly $h$ holes.  
Denoting this minimum by $g_{p,q}(h)$, we establish general lower and upper bounds, compute exact values in several small cases, and give a sufficient structural condition for a polyform with $h$ holes to attain $g_{p,q}(h)$.
\end{abstract}



\maketitle

\section{Introduction}

Using Schläfli's notation, a $\{p,q\}$-tessellation is an infinite tiling of the hyperbolic plane by regular $p$-gons (tiles), with $q$ such polygons meeting at each vertex. A \emph{polyform} is a finite collection of tiles of a $\{p,q\}$-tessellation whose interior is (point-set) connected. Tiles in a polyform are considered adjacent when they share an edge; tiles may meet only at a vertex, provided the polyform's edge connectivity is maintained by other tiles. $|A|$ denotes the number of tiles of a polyform $A$. By Alexander Duality, a \emph{hole} of a polyform is a bounded connected component of its complement. 
For example, Figure~\ref{fig:examples1} shows polyforms with six and three holes, respectively.  
To make large hyperbolic polyforms visually tractable, all figures in this paper were generated with Hypertiling~\cite{hypertiling}, which renders geodesic edges as straight segments in the chosen model.

\begin{defn}
Let $h$ be a positive integer. For given $p,q,h$, we define
\[
g_{p,q}(h) := \min \{\, |A| : A \text{ is a $\{p,q\}$-polyform with $h$ holes} \,\},
\]
that is, $g_{p,q}(h)$ denotes the minimum number of tiles required for a $\{p,q\}$-polyform to contain $h$ holes.  
We call a $\{p,q\}$-polyform with $h$ holes and $g_{p,q}(h)$ tiles \emph{crystallized}.
\end{defn}

Previous work on extremal topology in polyforms has focused on the Euclidean setting. The $\{4,4\}$- and $\{3,6\}$-tessellations consist of squares and equilateral triangles, respectively, whose polyforms are known as polyominoes and polyiamonds.
 In \cite{KaR} and \cite{MaR}, the function $g_{4,4}(h)$ was completely determined, together with explicit constructions of polyominoes with $g_{4,4}(h)$ tiles and $h$ holes. An analogous study was carried out in \cite{malen2021extremalI, malen2021extremalII}, where a full description of $g_{3,6}(h)$ was obtained, along with the corresponding polyiamonds realizing these values.  

In this paper, we extend this line of research to the hyperbolic setting. We study $g_{p,q}(h)$ for hyperbolic polyforms, providing exact values for certain small choices of $p,q$ and $h$, establishing general bounds on $g_{p,q}(h)$, and proposing conjectures about its asymptotic behavior.

\begin{figure}[htbp]
    \centering
    \includegraphics[width=0.4\linewidth]{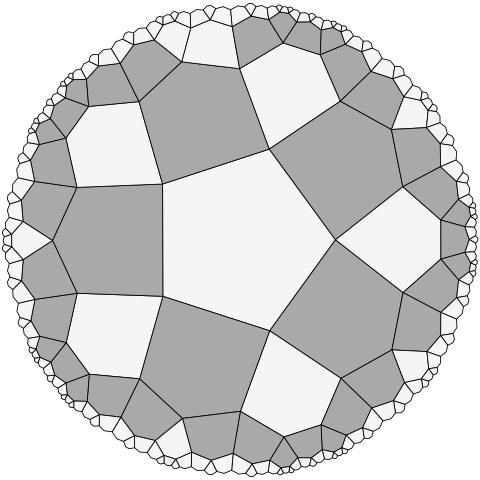}
    \includegraphics[width=0.4\linewidth]{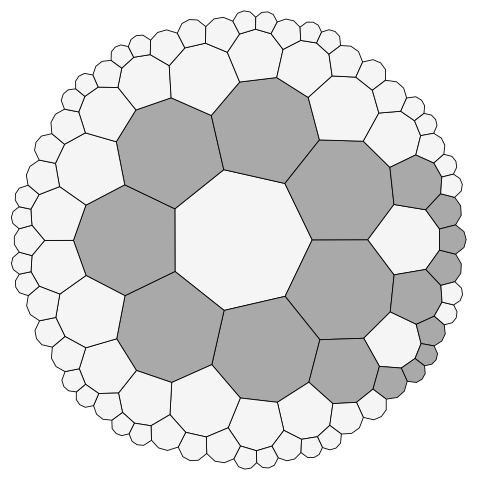}
    \caption{Examples of hyperbolic polyforms with multiple holes. 
Left: a $\{5,4\}$–polyform with six holes. 
Right: a $\{7,3\}$–polyform with three holes. 
Holes are defined as bounded connected components of the complement.}

    \label{fig:examples1}
\end{figure}

\subsection{Main Results}\label{sec:mainresults}

For the remainder of the paper, unless otherwise specified, we assume that $\{p,q\}$ corresponds to a hyperbolic tessellation, that is, $(p-2)(q-2) > 4$. 

Most of the functions and parameters introduced will depend on $p$ and $q$, but we omit the indices $p,q$ after their first appearance to simplify the notation. When necessary, we will explicitly indicate the dependence on $p,q$ to avoid ambiguity.

\begin{thm}\label{thm:table}
For every hyperbolic tessellation $\{p,q\}$, 
\[
g_{p,q}(1) = p+(p-1)(q-3) = pq-2p-q+3.
\]
Moreover, Table~\ref{table:g(h)vals} lists upper bounds on $g_{p,q}(h)$ for selected triples $(p,q,h)$; shaded entries are exact values.
\end{thm}

\begin{table}[h!]
    \centering
    \begin{tabular}{|c|c|c|c|c|c|c|c|c|}
    \hline
        $h$ & \{3,7\} &\{3,8\} & \{4,5\} & \{4,6\} & \{5,4\} & \{5,5\} & \{6,4\} & \{7,3\} \\
    \hline
    1   & \cellcolor{lightgray} \textbf{11}  & \cellcolor{lightgray} \textbf{13}      & \cellcolor{lightgray} \textbf{10}      & \cellcolor{lightgray} \textbf{13}      & \cellcolor{lightgray} \textbf{9}       & \cellcolor{lightgray} \textbf{13}      &  \cellcolor{lightgray} \textbf{11}  & \cellcolor{lightgray} \textbf{7} \\
    \hline
    2   & \cellcolor{lightgray} \textbf{18}  & \cellcolor{lightgray} \textbf{22}      & 17      & 23     & 15      & 23      & 19      & 12      \\
    \hline
    3   & \cellcolor{lightgray} \textbf{25}      & \cellcolor{lightgray} \textbf{31}      & 24      & 33      & 21      & 33      & 27      & 17      \\
    \hline
    4   & \cellcolor{lightgray} \textbf{31}      & \cellcolor{lightgray} \textbf{40}      & 31      & 42      & 27      &    43     &    35     &    22     \\
    \hline
    5   & 39      & \cellcolor{lightgray} \textbf{48}      & 38      &     52    & 33      &   53     &    43     &    26     \\
    \hline
    \end{tabular}
    \caption{Smallest known upper bounds of $g_{p,q}(h)$ obtained by explicit constructions,
for various $\{p,q\}$–tessellations and numbers of holes $h$.
Shaded entries correspond to crystallized polyforms.
}
\label{table:g(h)vals} 
\end{table}

To create this table, we found examples of holey polyforms using a program explained in Section \ref{section:table} and available in \cite{PolyformsGit}. The polyforms corresponding to shaded cells in Table~\ref{table:g(h)vals} satisfy the crystallization criterion of Theorem~\ref{thm:efficiently}; hence, those entries equal $g_{p,q}(h)$. Whether the unshaded cells represent crystallized polyforms remains an open question.

\begin{figure}[htbp]
    \centering
    
    \includegraphics[width=\polyscale]{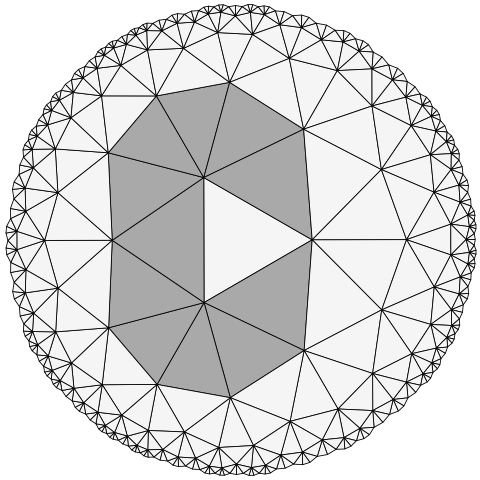}
    \includegraphics[width=\polyscale]{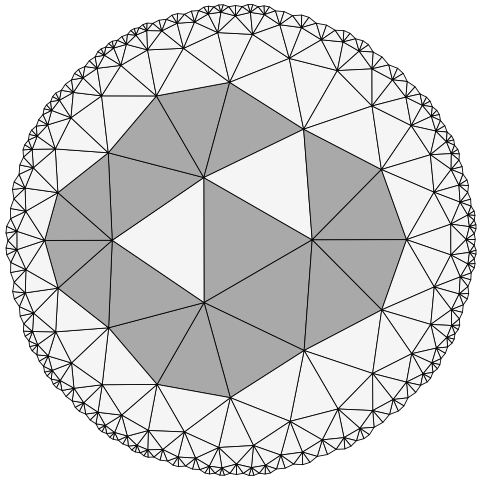}
    \includegraphics[width=\polyscale]{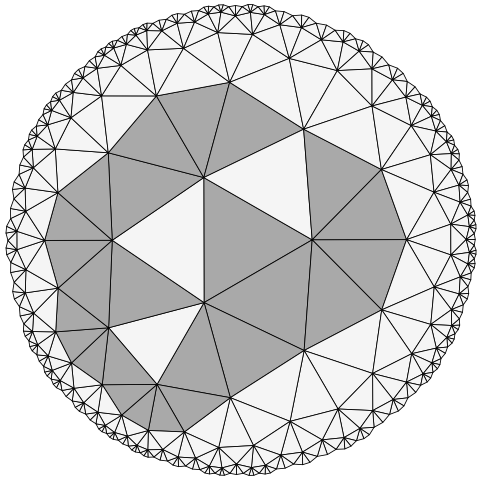}
    \includegraphics[width=\polyscale]{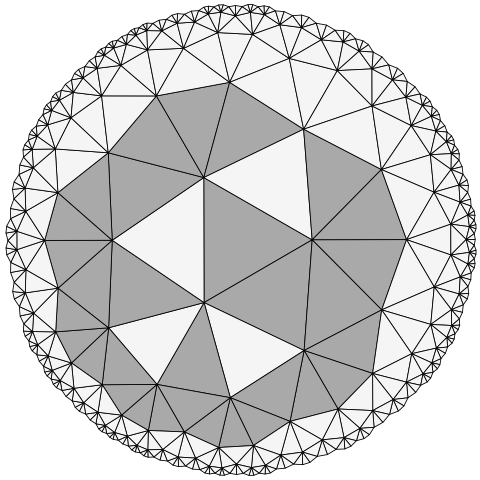}
    \includegraphics[width=\polyscale]{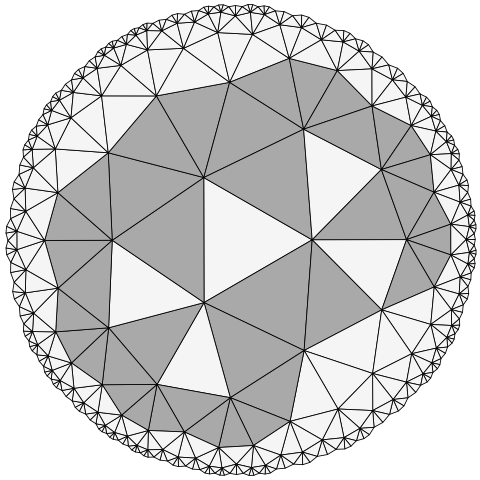}

    \includegraphics[width=\polyscale]{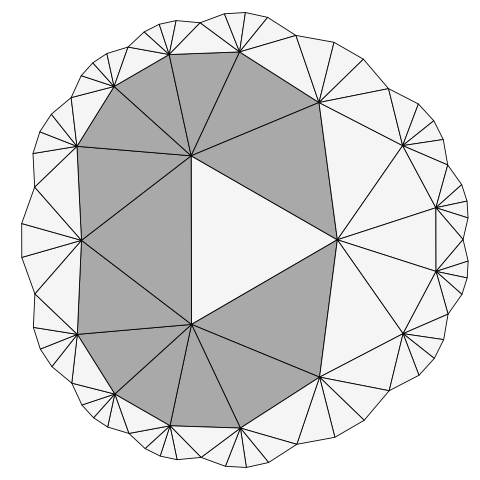}
    \includegraphics[width=\polyscale]{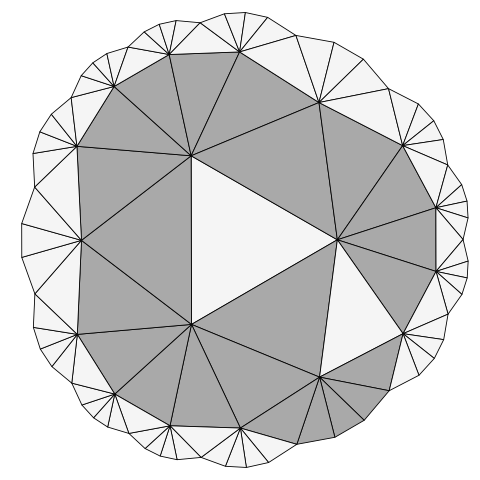}
    \includegraphics[width=\polyscale]{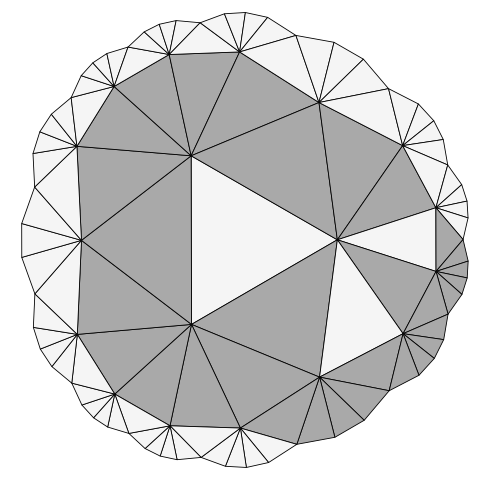}
    \includegraphics[width=\polyscale]{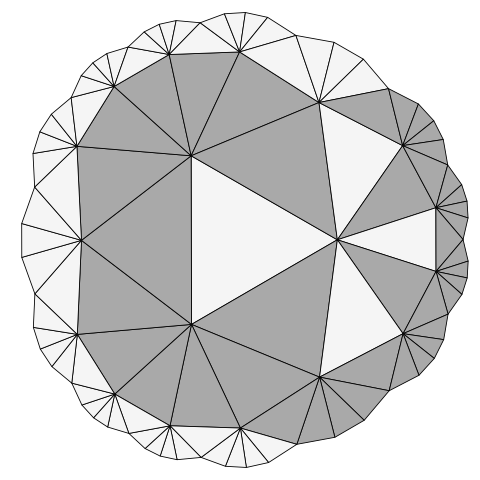}
    \includegraphics[width=\polyscale]{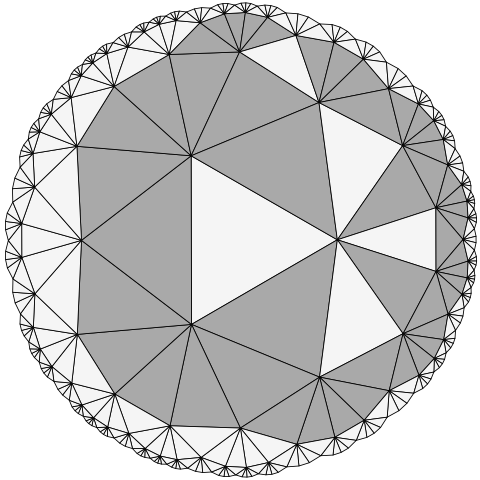}

    \includegraphics[width=\polyscale]{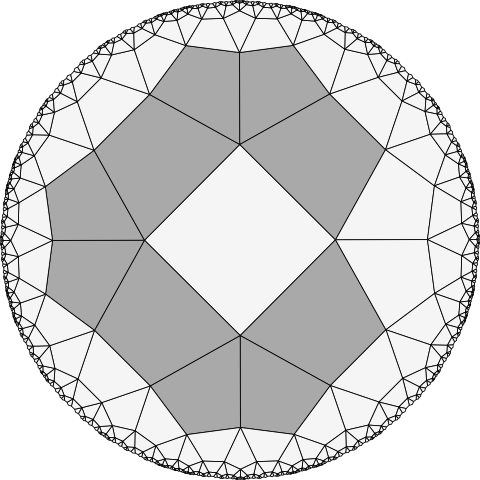}
    \includegraphics[width=\polyscale]{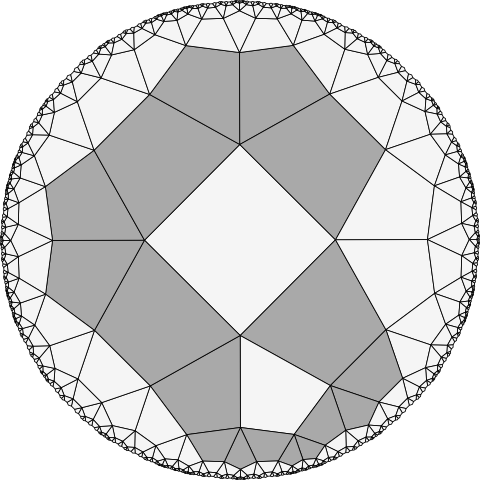}
    \includegraphics[width=\polyscale]{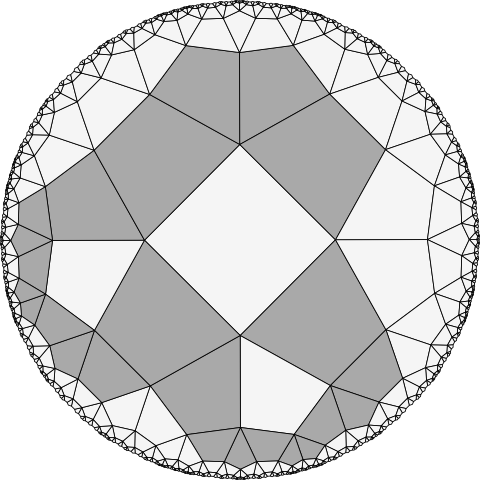}
    \includegraphics[width=\polyscale]{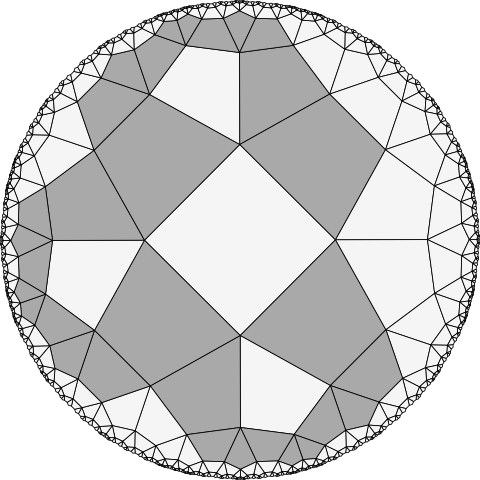}
    \includegraphics[width=\polyscale]{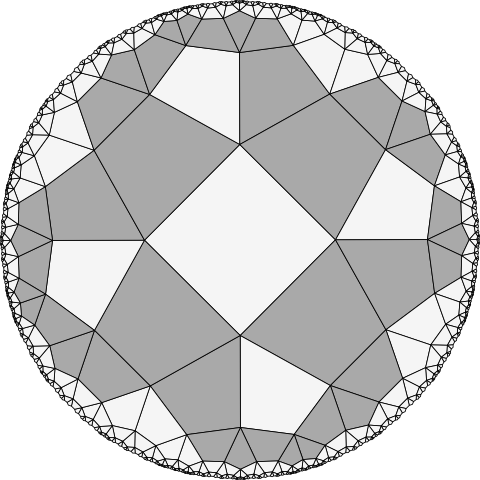}

    \includegraphics[width=\polyscale]{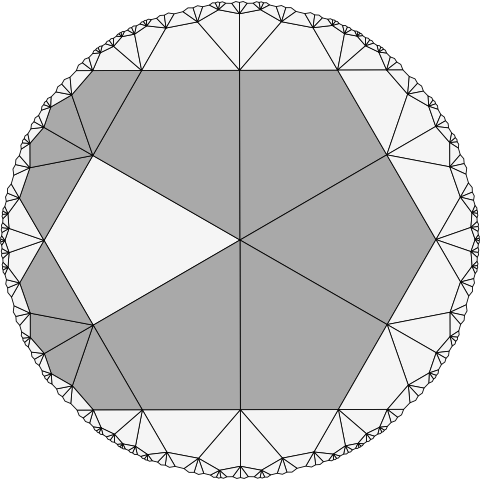}
    \includegraphics[width=\polyscale]{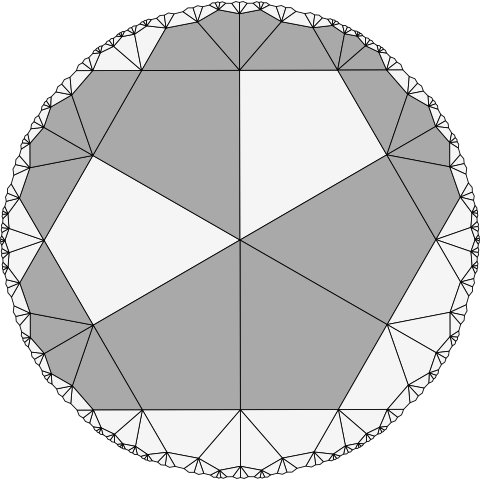}
    \includegraphics[width=\polyscale]{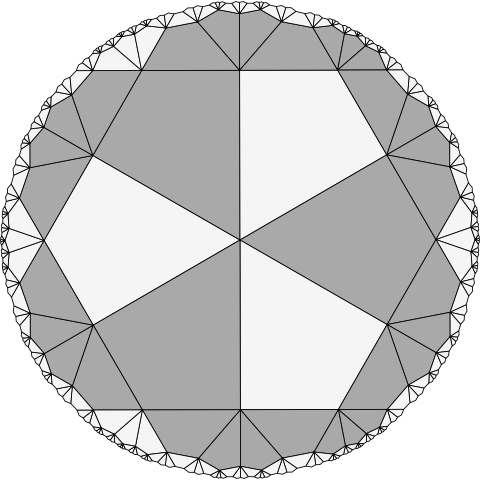}
    \includegraphics[width=\polyscale]{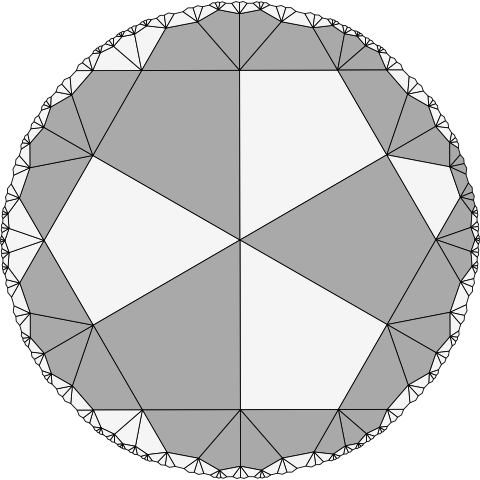}
    \includegraphics[width=\polyscale]{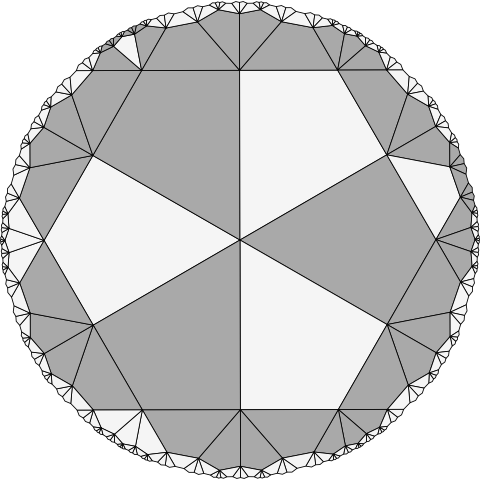}

    \includegraphics[width=\polyscale]{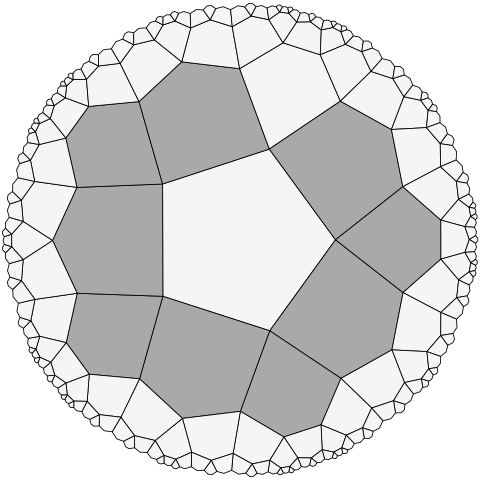}
    \includegraphics[width=\polyscale]{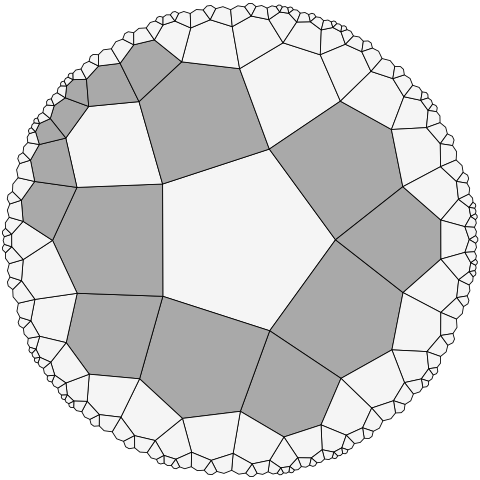}
    \includegraphics[width=\polyscale]{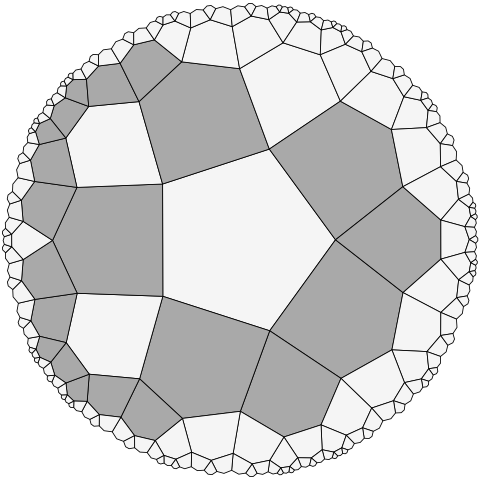}
    \includegraphics[width=\polyscale]{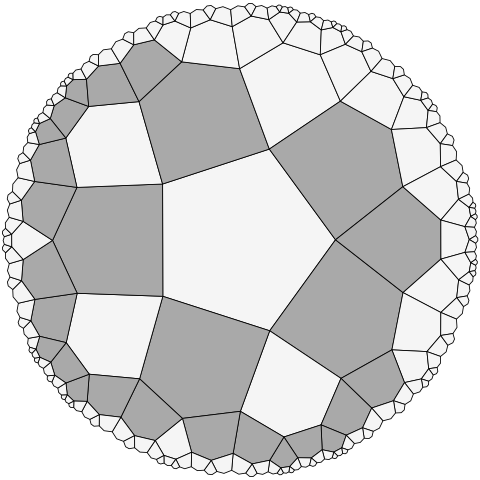}
    \includegraphics[width=\polyscale]{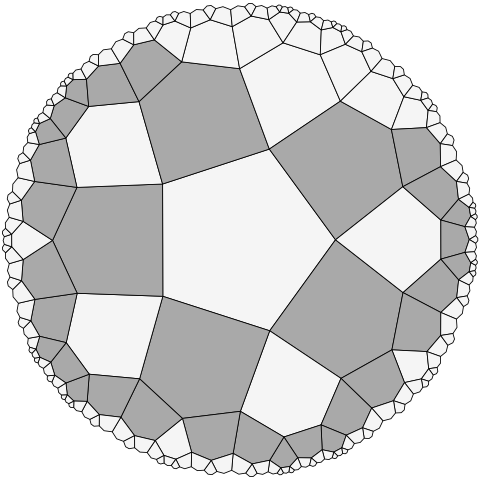}
    
    \includegraphics[width=\polyscale]{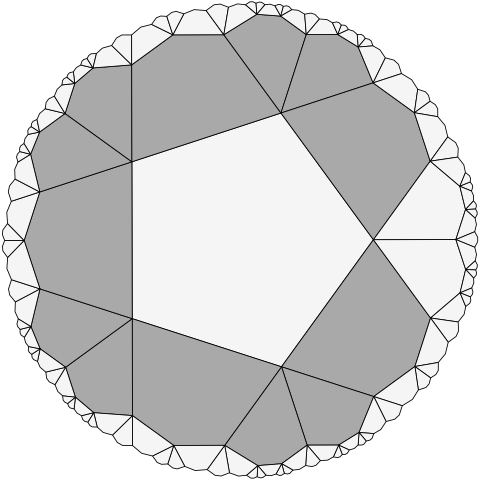}
    \includegraphics[width=\polyscale]{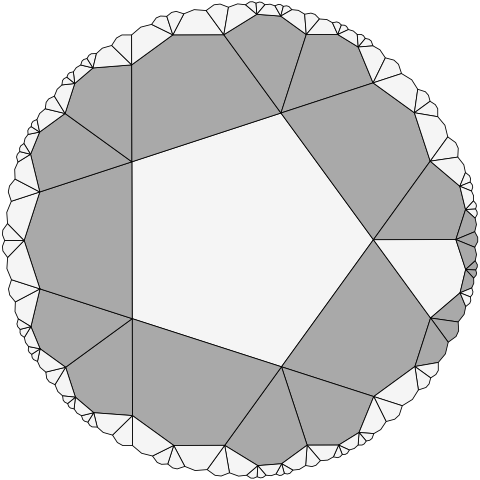}
    \includegraphics[width=\polyscale]{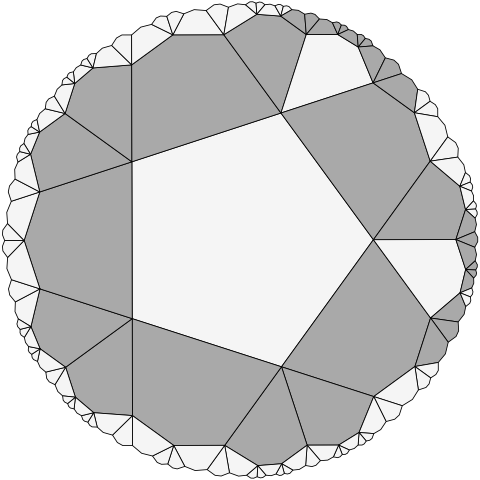}
    \includegraphics[width=\polyscale]{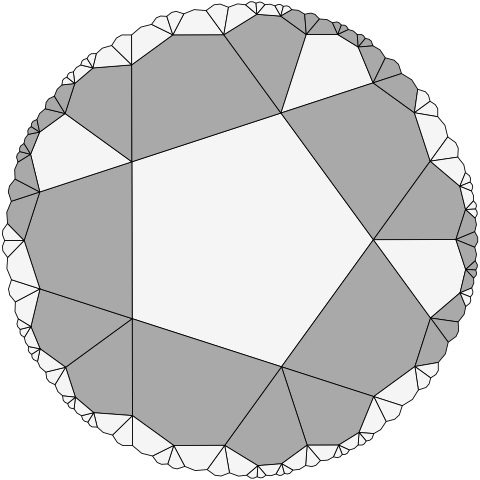}
    \includegraphics[width=\polyscale]{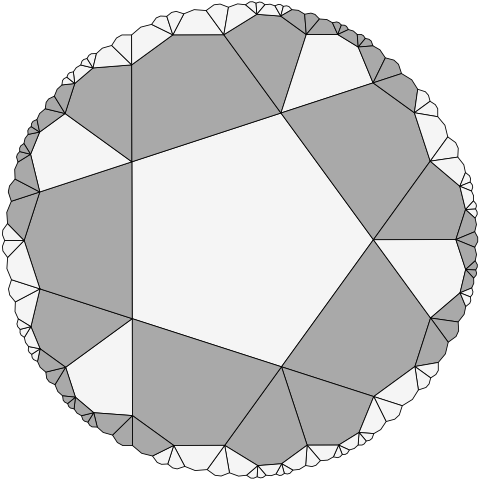}

    \includegraphics[width=\polyscale]{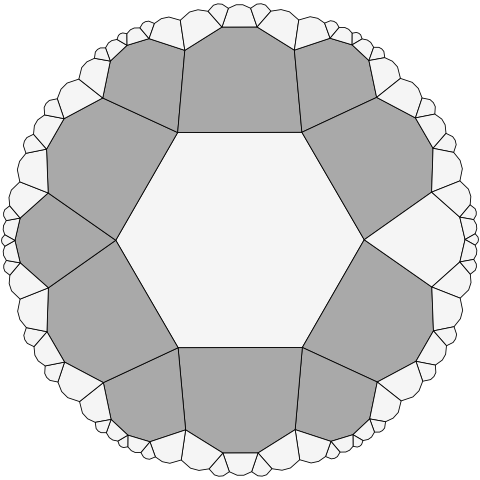}
    \includegraphics[width=\polyscale]{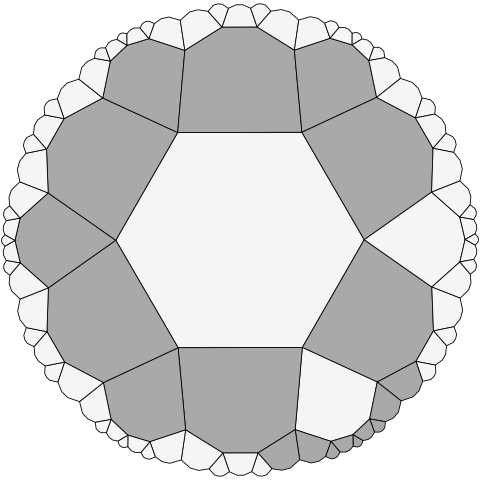}
    \includegraphics[width=\polyscale]{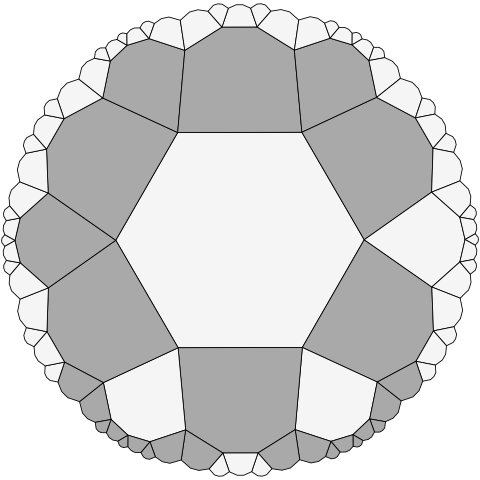}
    \includegraphics[width=\polyscale]{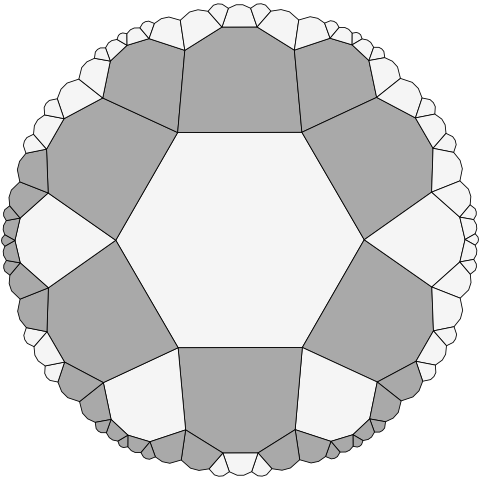}
    \includegraphics[width=\polyscale]{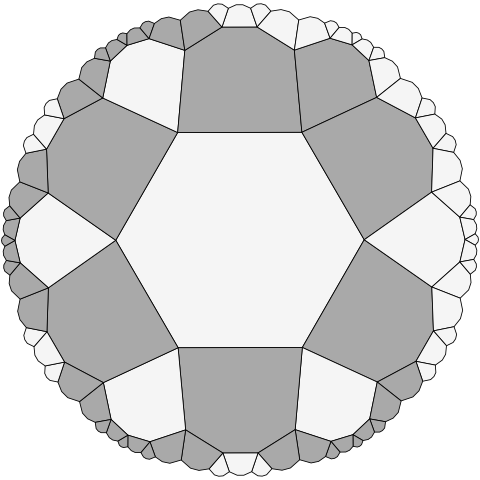}

    \includegraphics[width=\polyscale]{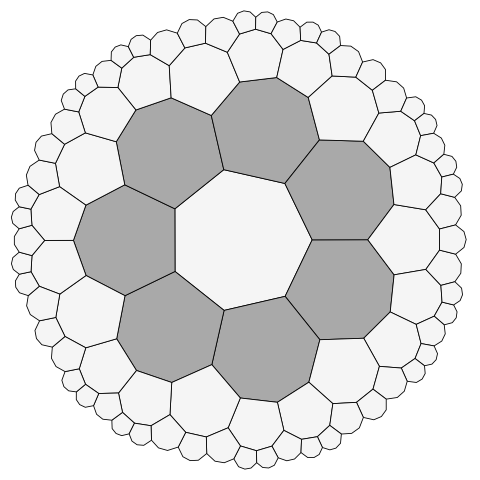}
    \includegraphics[width=\polyscale]{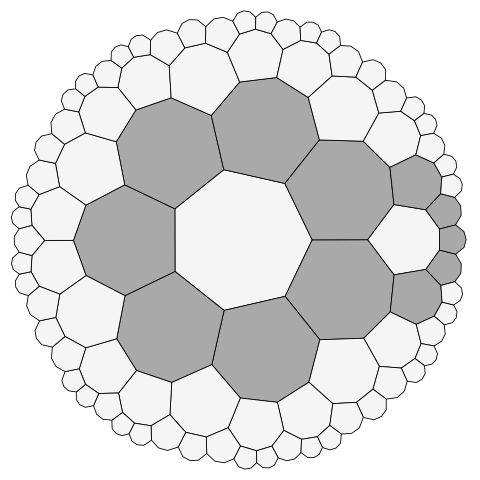}
    \includegraphics[width=\polyscale]{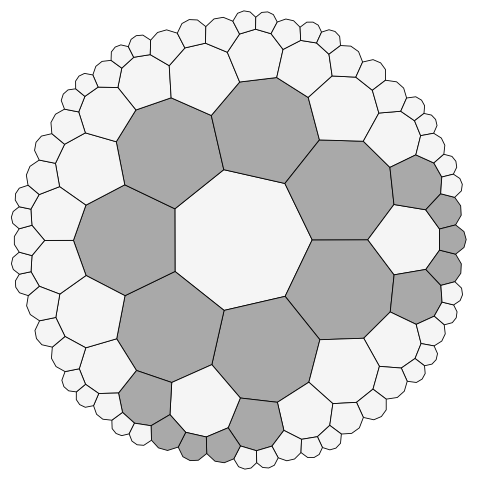}
    \includegraphics[width=\polyscale]{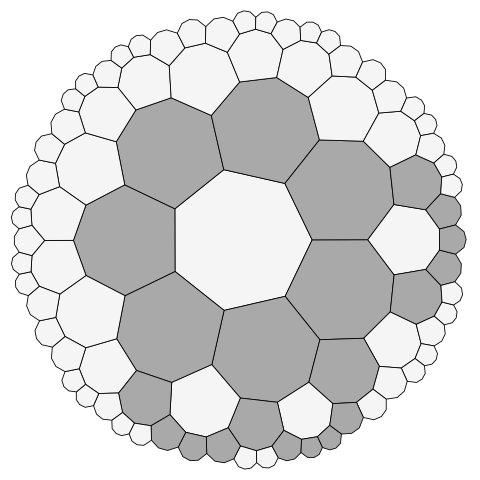}
    \includegraphics[width=\polyscale]{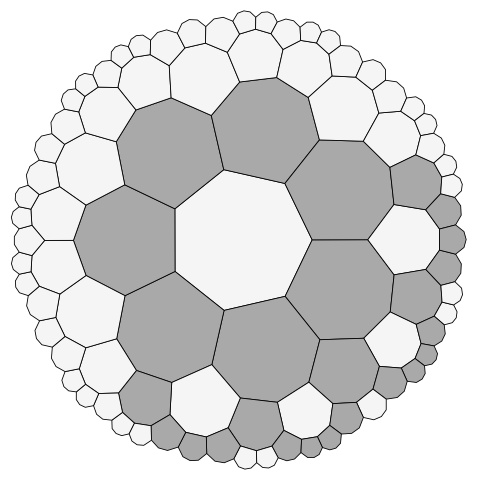} 

    \caption{Polyforms realizing the values in Table~\ref{table:g(h)vals},
arranged by tessellation (rows) and number of holes (columns).
} \label{fig:g(h)vals}
\end{figure}

We now introduce a sufficient condition for a polyform to be \emph{crystallized} (that is, to have $h$ holes and $g(h)$ tiles).

Intuitively, we seek to use tile edges as efficiently as possible to create holes. However, some edges must lie on the outer perimeter of the polyform (that is, edges adjacent to the unbounded region), while others must be shared between tiles to maintain connectivity. To formalize this intuition, we introduce the following concepts.

The \emph{dual graph} of a polyform $A$ is the graph $A'$ with one vertex for each tile of $A$, and with two vertices adjacent if and only if the corresponding tiles share an edge. We regard each hole of a polyform as a polyform in its own right, and define its \emph{area} as the number of tiles it contains. Let $H(A)$ denote the union of all hole regions of $A$, and let $\overline{A}$ be the polyform obtained by \emph{filling in} the holes of $A$, that is,
\[
\overline{A}=A\cup H(A).
\]

We call a polyform \emph{acyclic} if and only if its dual graph is a tree. The \emph{outer perimeter} of a polyform is the number of its edges that border the unbounded region. We say that a polyform is \emph{extremal} if it attains minimum perimeter among all polyforms with the same number of tiles.

\begin{defn}\label{def:eff}
Given a polyform $A$, we say that $A$ is \emph{efficiently structured} if and only if $A$ is acyclic, every hole in $A$ has area $1$ tile, and $\overline{A}$ is extremal. 
\end{defn}

To quantify efficient structure, we introduce the function $M(n,h)$, Equation \ref{eqn:M(n,h)} below, which measures the maximum number of holes achievable under perimeter constraints. This is explained in detail and proven in Lemmas \ref{lem:M(n,h)} and \ref{lem:equality}.

Let $\Pmin^{p,q}(N)$ denote the perimeter of an extremal polyform with $N$ tiles. As a generalization of the function of the same name from \cite{MaR}, we define $M_{p,q}(n,h)$ by:
     \begin{align}\label{eqn:M(n,h)}
         M_{p,q}(n,h) =  \frac{n(p-2)+2-\Pmin^{p,q}(n+h)}{p}.
     \end{align}

\begin{defn}\label{def:min-eff}
Fix $h\ge 1$. Let $A$ be a $\{p,q\}$-polyform with exactly $h$ holes, and set $n:=|A|$.
We say that $A$ is \emph{minimally efficiently structured (for $h$)} if
$A$ is efficiently structured and $n$ is the minimum integer such that $M(n,h)=h$.
\end{defn}

Note that this is strictly stronger than being efficiently structured: minimal efficient structure imposes, in addition, minimality of $n$ under the constraint $M(n,h)=h$. Thus, efficient structure is a structural condition, while minimal efficient structure is a fixed-$h$ extremal refinement.

\begin{thm} \label{thm:efficiently}
    If a polyform is minimally efficiently structured, then it is crystallized.
\end{thm}

Theorem \ref{thm:efficiently} provides the only known criterion for crystallization. We use this result to establish Theorem \ref{thm:table} and to derive a lower bound for  $g(h)$. A similar strategy was employed in \cite{KaR,malen2021extremalI} where an infinite sequence of minimally efficiently structured polyforms with $n_k$ tiles and $h_k$ holes was constructed for the cases $\{4,4\}$ (polyominoes) and $\{3,6\}$ (polyiamonds), respectively. These constructions later served as benchmarks to determine $g(h)$ for all $h$. For the case $\{3,6\}$,  a polyform is crystallized if and only if it is minimally efficiently structured, see \cite[Theorem 1.4]{malen2021extremalII}. In contrast, for  $\{4,4\}$, additional checkerboard obstructions were required to fully characterize crystallized polyforms, see \cite{MaR}, particularly Theorem 8.  We anticipate that the hyperbolic cases will exhibit features of both scenarios; however, at present, we lack sufficient evidence to confirm this.

The function $\Pmin^{p,q}$ plays a central role in characterizing minimally efficient structured polyforms. In \cite{MaRoTo, isoperimetric}, a detailed analysis of the function $\Pmin^{p,q}$ for polyforms is given. In particular, the asymptotic behavior of $\Pmin^{p,q}(n)$ is fine-tuned by a parameter $\beta$, which satisfies the quadratic equation $(p-2)\beta^2 -(p-2)(q-2)\beta + (q-2) = 0$. It is explicitly defined as follows: 
    \begin{align}   
                \beta_{p,q} =\frac{(p-2)(q-2)+\sqrt{(p-2)^2(q-2)^2-4(p-2)(q-2)}}{2(p-2)}. \label{eqn:beta}
    \end{align} 
In \cite[Theorem 1]{isoperimetric} it was proven that
\begin{align} \label{eqn:isoperimetric}
    \lim_{n\to \infty} \frac{\Pmin^{p,q}(n)}{n} = p-2-\frac{2}{\beta}.
\end{align}

Along similar lines, in \cite{d2025minimal}, the authors study minimal finite subgraphs of hyperbolic tessellations, focusing on graphs with a fixed number of vertices rather than tiles. They provide an explicit characterization of such graphs that achieve minimal perimeter, thereby determining exact Cheeger constants for graphs with a finite vertex set. Although their approach differs from ours—since we work with tiles instead of vertices—the connection is straightforward: the dual of a polyform with $n$ tiles in a $\{p,q\}$-tessellation is a connected graph with $n$ vertices in the dual $\{q,p\}$-tessellation. Crucially, the perimeters of these dual objects coincide. Consequently, the Cheeger (or isoperimetric) constant for a $\{p,q\}$-tessellation, as given in \cite{haggstrom2002explicit} and confirmed in \cite{d2025minimal}, is:
\begin{align}
    (q-2)\sqrt{1-\frac{4}{(p-2)(q-2)}}
\end{align}
This expression matches Equation (\ref{eqn:isoperimetric}) when $p$ and $q$ are interchanged.

In our third main result, we combine the expression for $M(n,h)$, the structural criterion from Theorem~\ref{thm:efficiently}, and the hyperbolic isoperimetric bounds for $\Pmin^{p,q}$ to obtain a general lower bound for $g(h)$. 
The corresponding upper bound is established by constructing explicit families of $\{p,q\}$--polyforms that realize $h$ holes with controlled size. 
The complete argument is presented in Section~\ref{section:bounds}.

\begin{thm} \label{thm:bounds} 
For every hyperbolic tessellation $\{p,q\}$ with $q>3$ and every integer $h \ge 2$, 
\begin{align} \label{eqn:bounds1}
    \beta\!\left(p-1-\frac{1}{\beta}\right) h  
    \;\le\;
    g_{p,q}(h)
    \;\le\;
    ((p-1)(q-2)+1)\, h .
\end{align}
In particular, $g_{p,q}(h)$ grows linearly with $h$, and the asymptotic ratio $g_{p,q}(h)/h$ is constrained between two explicit constants depending only on $p$ and $q$.
\end{thm}

\begin{rmk}
From \cite[Lemma 14]{isoperimetric} we know that:
\[
\lfloor \beta \rfloor = 
\begin{cases}
    q-4 , & \text{if } p=3, \\
    q-3 , & \text{if } p>3 .
\end{cases}
\]
This gives a clearer intuition for the sharpness of the bounds in \eqref{eqn:bounds1}. 
In particular, when $p,q>3$, Theorem~\ref{thm:bounds} yields
\[
   (p-1)(q-3)-1 
   \;\le\;
   \frac{g(h)}{h}
   \;\le\;
   (p-1)(q-2)+1,
\]
showing that the tile complexity per hole is tightly controlled in this range.
\end{rmk}

It is instructive to compare this result with the Euclidean cases. From \cite[Theorem 1.1]{malen2021extremalII} and \cite[Theorem 4]{MaR}, the following asymptotic behavior holds:
\begin{align} \label{eqn:g_euclid}
    g_{3,6}(h) = 3h + O(\sqrt{h}), \qquad
    g_{4,4}(h) = 2h + O(\sqrt{h}).
\end{align}

In the Euclidean setting, $\Pmin(N)$ grows as $O(\sqrt{N})$. Consequently, when solving the equation $M(n,h) = h$ for $n$, the leading term is $\frac{p}{p-2}h$, which matches the leading terms $3h$ and $2h$ for $p = 3$ and $p = 4$ in Equation (\ref{eqn:g_euclid}), respectively. The next order term is $O(\sqrt{h})$.

In contrast, for hyperbolic cases, $\Pmin(N)$ grows linearly as $O(N)$. Repeating the same procedure requires accounting for the linear contribution of $\Pmin(n+h)$. A detailed analysis then yields the lower bound stated in Theorem~\ref{thm:bounds}.

Our final result provides an upper bound for $g_{p,q}(h)$ when $p$ is even and $q=3$. The case $q=3$ is special because the tiles adjacent to a hole must also be adjacent to two other tiles adjacent to the hole (i.e., each hole is represented by a cycle in the dual graph). See Figure \ref{fig:examples1}. We use this property together with duality to reduce the problem of finding crystallized $\{2k,3\}$-polyforms with $h$ holes to finding extremal $\{2k,k\}$-polyforms with $h$ tiles. 

\begin{thm} \label{thm:q3}
For $k>3$ an integer, 
    \begin{align}
        g_{2k,3}(h) \leq  \frac{\Pmin^{2k,k}(h)+(2k-2)h}{2}+1.
    \end{align}    
\end{thm}

Based on the structure of the examples constructed to obtain the bound of Theorem \ref{thm:q3}, it is intuitively clear that the inequality is sharp. We conjecture that this is indeed the case. 
Combining this with the asymptotic behavior of $\Pmin^{p,q}(n)$, we would get the asymptotic behavior of $g_{2k,3}(h)$. 

\begin{conj} \label{cor:lim-q3}
    \begin{align}
    \lim_{h\to \infty} \frac{g_{2k,3}(h)}{h} = k-1+ \sqrt{(k-1)^2-\frac{2k-2}{k-2}}.
    \end{align}
\end{conj}

The rest of the paper is structured as follows: in Section \ref{sec:props} we prove various properties of crystallized hyperbolic polyforms and of the function $M(n,h)$. In Section \ref{section:table} we prove Theorem \ref{thm:table} using the techniques previously developed and also describe the algorithm to generate Table \ref{table:g(h)vals} and Figure \ref{fig:g(h)vals}. In Section  \ref{section:bounds}, we establish theoretical bounds on $g(h)$ by constructing examples of holey polyforms and applying a bound based on the minimal perimeter of a polyform, thereby proving Theorem \ref{thm:bounds}. Finally, in Section \ref{section:q3} we study the case $p$ even and $q=3$ and prove Theorem \ref{thm:q3}. Appendix \ref{appendix} summarizes the notation used throughout this paper.

\section{Properties of Crystallized Polyforms} \label{sec:props}

In this section we develop some tools necessary to study crystallized polyforms and we prove Theorem \ref{thm:efficiently} giving a sufficient condition for a polyform to be crystallized. Along the way we prove general properties of crystallized polyforms.

Recall that 
\begin{align}
    M_{p,q}(n,h) =  \frac{n(p-2)+2-\Pmin^{p,q}(n+h)}{p}.
\end{align}

 As usual, we omit the subindices $p,q$ (when unambiguous) to streamline notation.
  \begin{lem} \label{lem:M(n,h)}
    For all $n,h\ge 1 $,  if there is a polyform with $n$ tiles and $h$ holes, then
    \begin{align}  \label{eq:M(n,h)}
        h\le M(n,h).
    \end{align} 
    \end{lem}   
    \begin{proof}
The proof of this lemma follows the same logic as the derivation of $M_{4,4}(n,h)$ by Kahle and Rold\'an in \cite{KaR}, generalized to apply to an arbitrary hyperbolic $\{p, q\}$-tessellation. In their paper, for a $\{4,4\}$-polyform $A$, they define $p_o(A)$ to be the outer perimeter of $A$, $p_h(A)$ to be the hole perimeter of $A$ (that is, the number of edges in $A$ that border a bounded region of the complement of $A$), and $b(A)$ to be the number of edges in the dual graph of $A$. We define these similarly, except that $A$ can be any hyperbolic $\{p,q\}$-polyform. They also note that if $A$ has $n$ tiles and $h$ holes, the following hold:
    \begin{align}
         b(A) &\geq n-1 , \label{eq:ba} \\
         p_o(A) &= p_o(\overline{A}) \ge \Pmin^{p,q}(|\overline{A}|) \geq  \Pmin^{p,q}(n+h) , \textrm{ for } p>3. \label{eq:po}
    \end{align}
 The internal edges of $A$ correspond precisely to the edges of the graph $A'$. Since $A'$ is a connected graph with $n$ vertices, it has at least $n-1$ edges, which yields the first inequality.

The outer perimeter of $A$ coincides with the entire perimeter of the polyform $\overline{A}$, which contains at least $n+h$ tiles. Recalling that $\Pmin$ is non-decreasing for $p>3$, we obtain the second inequality.

Finally, for any tessellation $\{p,q\}$, observe that
\[
pn = 2b(A) + p_o(A) + p_h(A),
\]
since each tile has $p$ edges, and each edge either joins two tiles of $A$ or lies on its perimeter. Consequently,
\[
p_h(A) = pn - 2b(A) - p_o(A).
\]
Moreover, each hole contains at least one tile, and each tile has $p$ edges, so

\begin{align}
    h(A) \leq \frac{p_h(A)}{p} = \frac{pn-2b(A)-p_o(A)}{p}.
\end{align}
Using Equations \ref{eq:ba} and \ref{eq:po}, this leads to the inequality for $p>3$:
\begin{align}
    h(A) \leq \frac{pn-2(n-1)-\Pmin^{p,q}(|A|)}{p} \le \frac{n(p-2)+2-\Pmin^{p,q}(n+h)}{p}=M(n,h).
\end{align}
The case $p=3$ requires some adjustments regarding the bound on $p_o(A)$. The following argument, taken from \cite{MaR}, is adapted here to ensure it applies to the hyperbolic setting. We show first that $\Pmin^{3,q}(N)\le\Pmin^{3,q}(N+j)+1$ for all $N,j\ge 1$. Recall that the function $\Pmin^{3,q}(N)$ gives the perimeter of the spiral $\{3,q\}$-polyform with $N$ tiles. When a new tile is added to the spiral, the perimeter decreases by 1 or increases by 1, that is, $\Pmin^{3,q}(N)=\Pmin^{3,q}(N+1)\pm 1$. More importantly, a decrease by 1 is followed by at least $q-4$ increments by 1 (See \cite{MaRoTo}, Theorem 1.5 and the description of the sequences $d_{k,i}$ given below Definition 3.6). In particular, if $\Pmin^{3,q}(N) = \Pmin^{3,q}(N+1)+1$, then $\Pmin^{3,q}(N+1) = \Pmin^{3,q}(N+2)-1$. Therefore, $\Pmin^{3,q}(N) \le \Pmin^{3,q}(N+2)$ for all $N\ge1$. An induction argument on $j$ then shows $\Pmin^{3,q}(N)\le \Pmin^{3,q}(N+j)+1$ for all $N,j\ge 1$.  

Let $n+h+j=|\overline{A}|$, for some $j\ge0$. Then, $p_o(A)=p_o(\overline{A})\ge \Pmin^{3,q}(n+h+j)$. We still have that $b(A) \geq n-1$ and $p_h(A) = 3n-2b(A)-p_o(A)$, therefore,
\begin{align}
   p_h(A) & \le 3n-2(n-1)-\Pmin^{3,q}(n+h+j).
\end{align}
Holes in $\{3,q\}$-polyforms have at least 3 edges; so we can define $i=i(A)\ge 0$ such that $p_h(A)=3h+i$. If $j=i=0$ we get the desired result. Otherwise, $j>0$ implies $i>0$, and therefore:
\begin{align*}
   3h+i &\le 3n-2(n-1)-\Pmin^{3,q}(n+h+j), \\
   3h & \le  n+2-\Pmin^{3,q}(n+h+j)- i, \\
   3h & \le  n+2-\Pmin^{3,q}(n+h+j)- 1, \\
   3h & \le  n+2-\Pmin^{3,q}(n+h).
\end{align*}
This completes the proof.
    \end{proof}

\begin{cor} \label{cor:M(n,h)}
    For positive integers $h$ and $n$, if $h>M(n,h)$, then there is no polyform with $n$ tiles and $h$ holes. 
\end{cor}

\subsection{Efficient Structure}

In this section, we follow the ideas from \cite{MaR} to prove that a minimally efficient structured polyform is crystallized, when $q>3$. Note that for $q=3$ the dual graph $A'$ must have cycles, hence $A$ cannot be efficiently structured. 

    \begin{lem}  \label{lem:equality} Let $A$ be a polyform in a $\{p,q\}$ tessellation. If $A$ is efficiently structured, then $h = M(n,h)$.
    \end{lem}
    \begin{proof}
    For all polyforms, regardless of efficient structure, we have 
    $$pn = 2b(A) + p_0(A) + p_h(A).$$ 
    When $A$ is efficiently structured, by definition its dual graph is acyclic, every hole in \( A \) is minimal, and its outer perimeter is minimal. This translates to:
    \begin{itemize}
        \item $b(A) = n-1$, because the dual graph is a tree with $n$ vertices, therefore it has $n-1$ edges.
        \item $p_h(A) = ph$, because a minimal hole has area $1$ tile and therefore $p$ adjacent edges.
        \item $p_o(A) = \Pmin^{p,q}(n+h)$, because $\overline{A}$ is a polyform of minimal perimeter with $n+h$ tiles.
    \end{itemize}
    Plugging these into $pn = 2b(A) + p_0(A) + p_h(A)$, we get:
    \begin{align*}
     pn = 2(n - 1) + \Pmin^{p,q}(n + h) + ph.
    \end{align*}
Rearranging,
    \begin{align*}h = \frac{n(p - 2) + 2 -   \Pmin^{p,q}(n + h)}{p} = M(n, h).
    \end{align*}
    \end{proof}
    \begin{lem} \label{lem:Mincreasing} $M(n,h)$ is nondecreasing in $n$.
    \end{lem}
    \begin{proof}
    We have,
    \begin{align*}
        M(n+1, h) - M(n, h) = \frac{p - 2 - \left(\Pmin^{p,q}(n+1+h) - \Pmin^{p,q}(n+h)\right)}{p}.
    \end{align*}
    In \cite[Corollary 7.1]{MaRoTo}, Malen, Rold\'an, and Toala characterized the jumps of the function $\Pmin^{p,q}$. They found that :
\begin{align*}
&\Pmin^{p,q}(N+1)-\Pmin^{p,q}(N) = p-2 \text{ or $p-4$, when $q>3$.}  \\[4pt]
&\Pmin^{p,3}(N+1)-\Pmin^{p,3}(N)= p-4 \text{ or $p-6$, for $N>1$.}
\end{align*}   
    Thus, the expression $M(n+1, h) - M(n, h)$ is equal to either $0$, $\frac{2}{p}$, or $\frac{4}{p}$. It follows that $M(n,h)$ is nondecreasing. 
    \end{proof}
    
Let $A$ be a $\{p,q\}$-polyform with $n$ tiles and $h$ holes. Recall that Definition \ref{def:min-eff} states that $A$ is minimally efficiently structured if and only if $A$ is efficiently structured and $n$ is the least integer for which $M(n,h) = h$.
    
Note that $M(n,h)$ is nondecreasing in $n$ by Lemma \ref{lem:Mincreasing} and $A$ being efficiently structured implies $M(n,h) = h$ by Lemma \ref{lem:equality}. Therefore, it can be verified that a polyform $A$ is minimally efficiently structured by checking that $A$ is efficiently structured and $M(n-1,h) < h$. 
    
    \begin{thm*} [Theorem \ref{thm:efficiently}] \label{lem:g1-} Let \( A \) be a polyform in a \(\{p,q\}\)-tessellation. If \( A \) is minimally efficiently structured, then \( A \) is crystallized. 
    \end{thm*}
    \begin{proof}
    Let $A$ be a minimally efficiently structured $\{p,q\}$-polyform with $n$ tiles and $h$ holes. We know by Lemma \ref{lem:equality} that $ M(n,h) = h$. Moreover, by minimality and by Lemma \ref{lem:Mincreasing} we have that $M(m,h) < M(n,h)= h$, for all $m<n$. Finally, by Corollary \ref{cor:M(n,h)},  no $\{p,q\}$-polyform exists with $m$ tiles and $h$ holes. So, this implies that $n$ is the minimum number of tiles required for a $\{p,q\}$-polyform to have $h$ holes. In other words, $A$ is crystallized. 
    \end{proof}

\begin{rmk}
$M(n+1, h) = M(n, h)$ if and only if $\Pmin^{p,q}(n+1+h) - \Pmin^{p,q}(n+h)=p-2$. The increment by $p-2$ in $\Pmin$ occurs when the spiral $S_{p,q}(n+1+h)$ is formed by attaching the $(n+1+h)$-th tile to $S_{p,q}(n+h)$ along only 1 edge. If the tile is attached along 2 edges, the increment in $\Pmin$ is $p-4$ instead. Moreover, to construct the spiral polyforms, consecutive tiles can be attached along 1 edge at most $q-2$ times in succession. Consequently, for fixed $h$, the function $M(n,h)$ can remain constant for at most $q-2$ consecutive values of $n$. As a result, if there is a polyform with $n$ tiles, $h$ holes and $M(n,h)=h$, then $n-(q-2) \le g(h)\le n$.
\end{rmk}

\section{Some values of $g(h)$.} \label{section:table}

We begin by introducing two types of polyforms that act as reference points for our constructions.

\begin{defn}
Construct the polyform $A_{p,q}(k)$ inductively as follows: Let $A_{p,q}(1)$ denote the $\{p,q\}$-polyform with only one $p$-gon. Then, construct $A_{p,q}(k)$ from $A_{p,q}(k-1)$ by adding precisely the tiles needed so that all the perimeter vertices of $A_{p,q}(k-1)$ are surrounded by $q$ tiles. We call $A_{p,q}(k)$ the \emph{complete $k$-layered $\{p,q\}$-polyform}. See Definition 1.2 in \cite{MaRoTo} 
\end{defn}

\begin{defn}    
$S_{p,q}(n)$ is an extremal polyform with $n$ tiles. It is constructed by following a spiral pattern from a central tile. See Definition 4.1 and Theorem 1.8 in \cite{MaRoTo}.
\end{defn}

\begin{thm*} [First part of Theorem \ref{thm:table}]
    $g_{p,q}(1)=p +(p-1)(q-3)=pq-2p-q+3$.
\end{thm*}
\begin{proof}
    Construct the polyform $B=B_{p,q}(1)$ as follows: Begin with $A_{p,q}(2)$. Select a vertex incident to the central tile and remove all tiles adjacent to that vertex except those two that neighbor the central tile. See first column of Figure \ref{fig:g(h)vals}.
    
    $B$ is connected, because the second layer is unchanged except for the removal of $q-3$ consecutive tiles. Since the second layer initially formed a cycle, removing  consecutive tiles cannot disconnect it. $B$ also has a hole, as the tile that was in the first layer is now removed, but all tiles that were adjacent to it were preserved. 
    The number of tiles in $B$ is $p + (p-1)(q-3) = pq-2p-q+3$, as there are $p$ tiles adjacent to the hole, and $q-3$ tiles are required to connect two neighboring tiles adjacent to the same hole, one set for each of the $p-1$ vertices. 

    Finally, we compute $M(|B|-1,1)$. Equation (1.9) from \cite{isoperimetric} provides the formula for $\Pmin^{p,q}(N)$ for $N\le p(q-2)+1$. In particular, 
   \begin{align*}
        \Pmin^{p,q}(N) = 
              p + (p-2)\left(N-1 \right) -2 \left \lfloor \frac{N-2}{q-2}\right \rfloor, & \text{ if } q\le N \le p(q-2) .
    \end{align*}
    Plugging this into the formula $M(n,h) = \frac{n(p-2) + 2 -   \Pmin^{p,q}(n+h)}{p}$ and simplifying, it follows that for $n=pq-2p-q+2$ and $h=1$:
    \begin{align*}
        M(n,1) &= \frac{n(p-2)+2-\left( p+(p-2)(n+1-1)-2 \left \lfloor \frac{n-1}{q-2} \right  \rfloor \right)}{p}, \\
        &= \frac{-(p-2)+2 \left \lfloor \frac{p(q-2)-(q-2)-1}{q-2} \right \rfloor}{p}, \\
        &= \frac{-(p-2)+2 \left( p-2\right)}{p}= \frac{p-2}{p}<1. 
    \end{align*}
    Therefore, by Corollary \ref{cor:M(n,h)}, there is no polyform with $1$ hole and fewer tiles than $B$. Thus, $g(1) = pq-2p-q+3$.
\end{proof}

\begin{rmk}
Note that the value of $g_{p,q}(1)$ is equal to the upper bound established in Theorem \ref{thm:bounds}, therefore Theorem \ref{thm:bounds} holds for $h = 1$ as well.   
\end{rmk}

Now we proceed to complete the proof of Theorem \ref{thm:table} and describe the program used to obtain Table \ref{table:g(h)vals} and Figure \ref{fig:g(h)vals}.

The program \cite{PolyformsGit} first generates the subset of the $\{p,q\}$ tessellation equal to $A_{p,q}(3)$, then recursively generates polyforms one tile at a time in a manner similar to a depth-first search.
More specifically, the recursive step is a procedure that receives a polyform, the tile most recently added to the polyform (referred to as $X$), and a maximum number of tiles for the polyform as input, and that outputs a polyform with as many holes as possible with preference to minimal tiles. 

In the procedure, if the number of tiles in the polyform is less than the specified maximum, a copy of the polyform is created for each tile adjacent to $X$ that could be added to the polyform, and that tile is added to the corresponding copy. Each copy is then checked against the set of polyforms that have been generated so far, and any duplicates are discarded. The recursive step is then repeated on each of the copies, with each output replacing the respective input polyform. Then, the polyforms are compared and one with the most holes (and in the event of a tie, fewest tiles) is output. The program then begins by creating a polyform containing one tile and running the recursive procedure using it.

It is worth noting that this procedure does not enumerate all polyforms up to the given size. This is due to the starting tile having a fixed position and the requirement that tiles are only added to the polyform adjacent to the most recently added tile. These restrictions were added to allow a significant number of polyforms to be checked faster than if tiles could be added anywhere, while consuming less memory. Especially in cases where $p=3$, this can cause the output polyforms to not be crystallized. However, these polyforms appear to be only a few tiles larger than the crystallized ones, and they show that for many $p,q,$ and small $h$, $g(h) \leq g(1) + (h-1)(g(1)-3)$ except when $p=3$ or $q=3$. See Table \ref{table:g(h)vals} for a table of the fewest tiles we have found for a polyform with $h$ holes in various tessellations.   

It is also worth discussing why we only generated polyforms with a few holes. Despite sacrificing a fully exhaustive check in exchange for faster completion requiring less memory, the number of possible configurations for polyforms quickly becomes overwhelming. In order to cope with the number of possible structures, we must either enforce additional rules on how structures are generated or improve the detection of polyforms that are isomorphic to each other. We cannot enforce rules that have not been proven necessary for crystallization, as these assumptions may prevent the discovery of rare or unintuitive exceptions which could be required for crystallization in some cases. We could improve the detection of isomorphism, but this comes with more problems: Firstly, we must either identify a canon form for each polyform (which is an unsolved problem) or check every new polyform against every previously generated polyform, which would become increasingly expensive. Secondly, checking for isomorphism is not straightforward; enumeration of polyforms is also an unsolved problem, and simple techniques that help in the euclidean case are not immediately applicable to hyperbolic tessellations. Thirdly, these improvements would not make it feasible to check arbitrarily large polyforms, but would rather push back the cutoff point at which they take too long or consume too much memory to check. Ultimately this would not lead us to an equation for $g(h)$, so we directed our focus elsewhere.

\begin{table}[!ht]
\centering
\setlength{\tabcolsep}{4pt}
\renewcommand{\arraystretch}{1.15}
\small

\begin{tabular}{|c||c|c|c||c|c|c|}
\hline
 & \multicolumn{3}{c||}{\{3,7\}} & \multicolumn{3}{c|}{\{3,8\}} \\
\hline
$h$ & $M(n\!-\!2,h)$ & $M(n\!-\!1,h)$ & $M(n,h)$ & $M(n\!-\!2,h)$ & $M(n\!-\!1,h)$ & $M(n,h)$ \\
\hline
1 & 0.3 & \crystcell{0.3} & \crystcell{1.0} & 0.3 & \crystcell{0.3} & \crystcell{1.0} \\
\hline
2 & 1.3 & \crystcell{1.3} & \crystcell{2.0} & 1.3 & \crystcell{1.3} & \crystcell{2.0} \\
\hline
3 & 2.3 & \crystcell{2.3} & \crystcell{3.0} & 2.3 & \crystcell{2.3} & \crystcell{3.0} \\
\hline
4 & 3.3 & \crystcell{3.3} & \crystcell{4.0} & 3.3 & \crystcell{3.3} & \crystcell{4.0} \\
\hline
5 & 4.3 & \precrcell{5.0} & 5.0 & 4.3 & \crystcell{4.3} & \crystcell{5.0} \\
\hline
\end{tabular}

\vspace{6pt}

\begin{tabular}{|c||c|c|c||c|c|c|}
\hline
 & \multicolumn{3}{c||}{\{4,5\}} & \multicolumn{3}{c|}{\{4,6\}} \\
\hline
$h$ & $M(n\!-\!2,h)$ & $M(n\!-\!1,h)$ & $M(n,h)$ & $M(n\!-\!2,h)$ & $M(n\!-\!1,h)$ & $M(n,h)$ \\
\hline
1 & 0.5 & \crystcell{0.5} & \crystcell{1.0} & 0.5 & \crystcell{0.5} & \crystcell{1.0} \\
\hline
2 & 1.5 & \precrcell{2.0} & 2.0 & 1.5 & \precrcell{2.0} & 2.0 \\
\hline
3 & 2.5 & \precrcell{3.0} & 3.0 & 2.5 & \precrcell{3.0} & 3.0 \\
\hline
4 & 4.0 & 4.0 & 4.5 & 3.5 & \precrcell{4.0} & 4.0 \\
\hline
5 & 5.0 & 5.5 & 5.5 & 5.0 & 5.0 & 5.0 \\
\hline
\end{tabular}

\vspace{6pt}

\begin{tabular}{|c||c|c|c||c|c|c|}
\hline
 & \multicolumn{3}{c||}{\{5,4\}} & \multicolumn{3}{c|}{\{5,5\}} \\
\hline
$h$ & $M(n\!-\!2,h)$ & $M(n\!-\!1,h)$ & $M(n,h)$ & $M(n\!-\!2,h)$ & $M(n\!-\!1,h)$ & $M(n,h)$ \\
\hline
1 & 0.6 & \crystcell{0.6} & \crystcell{1.0} & 0.6 & \crystcell{0.6} & \crystcell{1.0} \\
\hline
2 & 1.6 & \precrcell{2.0} & 2.0 & 1.6 & \precrcell{2.0} & 2.0 \\
\hline
3 & 2.6 & \precrcell{3.0} & 3.4 & 2.6 & \precrcell{3.0} & 3.0 \\
\hline
4 & 4.0 & 4.0 & 4.4 & 4.0 & 4.0 & 4.0 \\
\hline
5 & 5.0 & 5.4 & 5.4 & 5.0 & 5.0 & 5.0 \\
\hline
\end{tabular}

\vspace{6pt}

\begin{tabular}{|c||c|c|c||c|c|c|}
\hline
 & \multicolumn{3}{c||}{\{6,4\}} & \multicolumn{3}{c|}{\{7,3\}} \\
\hline
$h$ & $M(n\!-\!2,h)$ & $M(n\!-\!1,h)$ & $M(n,h)$ & $M(n\!-\!2,h)$ & $M(n\!-\!1,h)$ & $M(n,h)$ \\
\hline
1 & 0.7 & \crystcell{0.7} & \crystcell{1.0} & 0.4 & \crystcell{0.7} & \crystcell{1.3} \\
\hline
2 & 1.7 & \precrcell{2.0} & 2.0 & 2.0 & 2.3 & 2.6 \\
\hline
3 & 2.7 & \precrcell{3.0} & 3.0 & 3.6 & 3.9 & 4.1 \\
\hline
4 & 3.7 & \precrcell{4.0} & 4.3 & 5.1 & 5.4 & 5.7 \\
\hline
5 & 5.0 & 5.0 & 5.3 & 6.4 & 6.7 & 7.0 \\
\hline
\end{tabular}

\caption{Computed values of $M(n-2,h)$, $M(n-1,h)$, and $M(n,h)$ for the parameter tuples $(p,q,h,n)$
appearing in Table~\ref{table:g(h)vals}, where $n$ is the corresponding entry in that table.
Dark gray shading certifies crystallization at $n$, i.e.\ $M(n-1,h) < h \le M(n,h)$.
Light gray shading indicates that $n-2$ tiles are ruled out by the $M$--bound while $n-1$ tiles are not,
i.e.\ $M(n-2,h) < h \le M(n-1,h)$.
}
\label{table:g(h)values-proof}
\end{table}

To conclude this section, we show that shaded entries in Table~\ref{table:g(h)vals} correspond to crystallized polyforms. This is verified by checking that $h > M(n-1,h)$ and applying Corollary~\ref{cor:M(n,h)}. Thus completing the proof of Theorem~\ref{thm:table}.

Table~\ref{table:g(h)vals} lists the number of tiles $n$ of the various $\{p,q\}$-polyforms with $h$ holes depicted in Figure \ref{fig:g(h)vals}. For each tuple $(p,q,h,n)$, we compute $M(n-2,h)$, $M(n-1,h)$, and $M(n,h)$. The results are summarized in Table~\ref{table:g(h)values-proof}. The values of $\Pmin$ are obtained using Theorem~1 in \cite{isoperimetric}; an online calculator is available in \cite{beautifulanimals2}.

\section{Proof of Theorem \ref{thm:bounds}} \label{section:bounds}

In this section we show matching linear bounds on $g_{p,q}(h)$. We combine the techniques for polyominoes from \cite{MaR} with  Rold\'an--Toalá’s isoperimetric Theorem 1 in \cite{isoperimetric} to obtain the lower bound $(\beta p-\beta-1)h$ and use an explicit construction to get the upper bound $((p-1)(q-2) + 1)h$. 

For the explicit construction, we start by reviewing properties of hyperbolic geometry necessary to create the construction. 

\begin{lem} \label{parallel quadrilaterals}
    If $l$ and $m$ are geodesics that pass through opposing edges of a hyperbolic quadrilateral with all interior angles less than $\frac{\pi}{2}$, then $l$ and $m$ are ultraparallel \footnote{In Hyperbolic Geometry, two lines are said to be ultraparallel if they do not intersect, even asymptotically (i.e., even at the boundary of the Poincaré disk model).}.
\end{lem}
\begin{proof}
    Let $Q$ be a hyperbolic quadrilateral with all interior angles less than $\frac{\pi}{2}$. Label the geodesics that pass through each edge of $Q$ as $l, m, l',$ and $m'$, with $l,m$ and $l',m'$ passing through opposing edges. Suppose for the sake of contradiction that $l$ and $m$ intersect at a point $x$.
    Note that $l'$ and $m'$ each partition the plane into two regions. Of the two regions that $l'$ partitions, assume $x$ is located in the region not containing the quadrilateral $Q$. Because the interior angles of $Q$ are all less than $\pi/2$, the triangle formed by $x$ and the intersection of $l'$ with $l$ and $m$ must have angles greater than $\pi/2$ at two of its corner. Thus, the sum of its interior angles is greater than $\pi$, which is a contradiction on the hyperbolic plane, and $x$ must be located in the region containing $Q$. For the same reason it must be on the side of $m'$ containing $Q$. However, if $l$ and $m$ intersect in this region, there exists a point on $m$ between the vertices of $Q$ where $m$ and $l'$ meet and where $m$ and $m'$ meet, which contradicts the hypothesis that $Q$ is a quadrilateral. Thus, $l$ and $m$ are ultraparallel. The same argument proves that $m'$ and $l'$ are ultraparallel.
\end{proof}
\begin{lem}\label{parallel tile edges}
    If $X$ is a tile in a $\{p,q\}$-tessellation with $q > 3$ then geodesics through nonadjacent edges of $X$ are ultraparallel.
\end{lem}
\begin{proof}
     Let $X$ be a tile in a $\{p,q\}$-tessellation with $q > 3$. Because $q > 3$ all interior angles of $X$ are less than or equal to $\frac{\pi}{2}$. Let $l$ and $m$ be the non-adjacent sides of $X$. Connect $l$ and $m$ by two geodesics, each going through one vertex on $l$ and one on $m$ to form a quadrilateral. Because $X$ is convex, the angles formed by these geodesics and $l$ and $m$ are less than or equal to the interior angles of each vertex, which is in turn no greater than $\frac{\pi}{2}$. This implies that for any pair of geodesics through non-adjacent edges of $X$ either there is a perpendicular between them, or the quadrilateral formed by them and geodesics through their vertices has all angles less than $\frac{\pi}{2}$. In the first case it is clear that they are ultraparallel, and in the second it is implied by Lemma \ref{parallel quadrilaterals}.  
\end{proof}

\begin{lem}\label{parallel tile edgesp3}
    If $p=3$, then opposite sides of a quadrilateral formed by two adjacent tiles are ultra parallel.
\end{lem}
\begin{proof}
This can be shown by contradiction. Note that each quadrilateral of this form has two opposing interior angles of $\frac{2\pi}{q}$ and two opposing interior angles of $\frac{4\pi}{q}$. If there exists any point at which these lines intersect, they would form a triangle with one side of the quadrilateral and the extension of the two sides adjacent to it. This triangle would have one angle measuring $\pi - \frac{2\pi}{q}$ and one measuring $\pi - \frac{4\pi}{q}$. Thus, the interior angles of this triangle would add to at least $2\pi - \frac{6\pi}{q}$. Since $q \ge 7$ when $p = 3$, the angles sum to no less than $2\pi - \frac{6\pi}{7} = \frac{8\pi}{7} > \pi$, which is a contradiction as hyperbolic triangles have an interior angle sum of no more than $\pi$. Thus, opposing sides are parallel.
\end{proof}
This concludes the preparations to construct the sequence of polyforms for Theorem \ref{thm:bounds}. Now we state a lemma to bound the number of tiles needed to achieve 2 holes. 

\begin{lem} \label{lem:h2}
    If $n+h\le p(q-2)+1$ and $h\ge 2$, then there is no polyform with $n$ tiles and $h$ holes. 
\end{lem}
\begin{proof}
    By Corolary \ref{cor:M(n,h)}, it suffices to show that $h>M(n,h)$. For values $N\le p(q-2)+1$ we have the formula for $\Pmin^{p,q}(N)$ from \cite[Equation (1.9)]{isoperimetric}, .
   \begin{align}   \label{eqn:P_min3}
        \Pmin^{p,q}(N) = 
        \begin{cases}
              p+(p-2)(N-1), & \text{ if } 1\le N < q, \\
              p + (p-2)\left(N-1 \right) -2 \left \lfloor \frac{N-2}{q-2}\right \rfloor, & \text{ if } q\le N \le p(q-2) , \\
              p(p-2)(q-2)-p, & \text{ if } N = p(q-2)+1. \\
        \end{cases}
    \end{align}
    Plugging this into the formula $M(n,h) = \frac{n(p-2) + 2 -   \Pmin^{p,q}(n+h)}{p}$ and simplifying, it follows that for $n+h\le p(q-2)+1$:
    \begin{align*}
M(n,h) = \begin{cases}
             \frac{-h(p-2)}{p}, & \text{ if } 1\le n+h < q, \\
              \frac{-h(p-2)+2 \left \lfloor \frac{n+h-2}{q-2}\right \rfloor }{p}, & \text{ if } q\le n+h \le p(q-2) , \\
              \frac{-h(p-2)+2p}{p}, & \text{ if } n+h = p(q-2)+1. \\
        \end{cases}
\end{align*}
In all cases
    \begin{align*}
M(n,h) \le \frac{-h(p-2)+2p}{p}<2 \le h .
\end{align*}
Therefore, any polyform with $h\ge2$ must satisfy $n+h > p(q-2)+1$.
\end{proof}

\begin{thm*}[Theorem \ref{thm:bounds}]
For $q>3$ and any integer $h\geq 2$,
    \begin{align} \label{eqn:bounds}
      \beta\left(p-1-\frac{1}{\beta}\right)h  \leq g(h)  \leq h((p-1)(q-2) + 1).
    \end{align}
\end{thm*}
    
    \begin{proof}
First, we prove the lower bound. Using Lemma \ref{lem:M(n,h)}, we have:
\begin{align*}
h \leq M(n,h) = \frac{n(p-2) + 2 -   \Pmin^{p,q}(n+h)}{p}.
\end{align*}
Since $g(h)$ is by definition the smallest number of tiles $n$ such that a polyform with $h$ holes exists, we may substitute $n = g(h)$:
\begin{align} \label{eqn:h<M}
h \leq \frac{g(h)(p-2) + 2 -   \Pmin^{p,q}(g(h)+h)}{p}.
\end{align}

From the isoperimetric formula in  \cite[Theorems 1 and 2]{isoperimetric} we have for $N>p(q-2)+1$ that
\begin{align}\label{eqn:Pmin}
\Pmin^{p,q}(N) =  \left(p - 2 - \frac{2}{\beta}\right) N + \epsilon(N) >   \left(p - 2 - \frac{2}{\beta}\right) N+2 .
\end{align}

By Lemma \ref{lem:h2}, we have $n+h>p(q-2)+1$. Therefore, we can apply the bound \ref{eqn:Pmin} on \ref{eqn:h<M} to get:
\begin{align*}
h \leq \frac{g(h)(p-2) + 2 - \left(p - 2 - \frac{2}{\beta}\right) (g(h)+h)-2}{p}.
\end{align*}
\\
Through basic algebraic manipulations of the above, we get
\begin{align*}
   hp &\leq g(h)(p-2) - \left(p-2-\frac{2}{\beta}\right)(g(h)+h),\\
   hp &\leq g(h)(p-2) - g(h)\left(p-2-\frac{2}{\beta}\right) - h\left(p-2-\frac{2}{\beta}\right),\\
   hp + h\left(p-2-\frac{2}{\beta}\right) &\leq g(h)\left(p-2 - \left(p-2-\frac{2}{\beta}\right)\right),\\
   2h\left(p-1-\frac{1}{\beta}\right) &\leq \frac{2g(h)}{\beta}.
\end{align*}
Because $\beta$ is strictly positive, we can divide through by $\frac{2}{\beta}$ to get
\begin{align*}
\beta\left(p-1-\frac{1}{\beta}\right)h \leq g(h).
\end{align*}

To prove the upper bound of $h((p-1)(q-2) + 1)$ we construct a family of sequences of $\{p,q\}$-polyforms with $h$ holes, which we label $B_{p,q}(h)$. 

We first create $B_{p,q}(1)$ by selecting some tile $X$ in the $\{p,q\}$-tessellation as a hole, selecting one vertex of $X$ and including all tiles on each other vertex of $X$ into $B_{p,q}(1)$, as shown in Figure \ref{fig:stack} below. This creates a connected figure with one hole. Since all tiles adjacent to $X$ share two vertices with it, and including the tiles on the vertex shared by any two tiles adjacent to $X$ connects them. Thus, this creates a polyform with $p + (p-1)(q-3) = (p-1)(q-2) + 1$ tiles. To create $B_{p,q}(h)$, we select two ultraparallel edges of $B_{p,q}(1)$ and reflect along them to create copies of it across these edges. The rest of the proof explains this construction in detail.

If $p > 3$, let $l_1, l_2$ be two nonadjacent edges of $X$ (this is possible because $X$ has at least 4 edges); by Lemma \ref{parallel tile edges}, they are ultra-parallel. Moreover, they divide the hyperbolic plane into 3 regions: One strip between $l_1$ and $l_2$, where $X$ lies; and two half-planes $D_1$ and $D_2$, each having $l_1$ and $l_2$ as boundaries, respectively. Let $l$ be the reflection of $l_1$ across $l_2$ and let $m$ be the reflection of $l_2$ across $l_1$. By construction, $l$ and $m$ are edges of tiles adyacent to $X$. Moreover, $l$ and $m$ are ultraparallel, as they are separated by the strip between $l_1$ and $l_2$.This gives us a pair of ultraparallel edges on the outer perimeter of $B_{p,q}(1)$. See figure \ref{fig:stack}.

When $p = 3$, let $Y_1$ and $Y_2$ be tiles adjacent to $X$. Consider the quadrilateral formed by $X$ and $Y_1$, let $l_1$ and $l_2$ be opposite edges of this quadrilateral such that they are not on the outer perimeter of $B_{p,q}(1)$ and with $l_1$ edge of $Y_1$, $l_2$ edge of $X$. By Lemma \ref{parallel tile edgesp3}, $l_1$ and $l_2$ are ultraparallel. As before, they split the hyperbolic plane in 3 regions, the strip where $X$ lies and two half-planes $D_1$ and $D_2$. Without loss of generality, $Y_2$ is in $D_2$ and by connectivity, there is another tile of $B_{p,q}(1)$ adjacent to $Y_2$, let $l$ be the edge opposite to $l_2$ in the quadrilateral formed by $Y_2$ and this adjacent tile. Note that $l_2$ lies entirely in $D_2$, since $l$ and $l_2$ are ultraparallel. Similarly, there are another 2 tiles on $D_1\cap B_{p,q}(1)$, one of them connected to $Y_1$ along $l_1$. These 2 tiles form another quadrilateral; let $m$ be the edge opposite to $l_1$. Note that $m$ lies entirely on $D_1$. Finally, $l$ and $m$ the desired ultraparallel edges on the outer perimeter of $B_{p,q}(1)$. See figure \ref{fig:stack}.

To create $B_{p,q}(2)$, we simply select one of these parallel edges and include the image of $B_{p,q}(1)$ reflected across the geodesic through this edge into the polyform. Since we are reflecting across an edge of a tile the polyform remains connected, and there is no self intersection because the new copy lies entirely on one side of the geodesic through this edge. By repeating this process, we can create a polyform $B_{p,q}(h)$ with exactly $h(p-1)(q-2) + h$ tiles since each copy of $B_{p,q}(1)$ is a self contained structure with one hole and $(p-1)(q-2) + 1$ tiles. The shape $B_{p,q}(h)$ does not self-intersect as the copies of $B_{p,q}(1)$ lie in disjoint strips bounded by two ultraparallel lines. Moreover, adjacency along boundary edges guarantees connectivity.

\begin{figure}[htbp]
    \centering
    \includegraphics[width=0.45\linewidth]{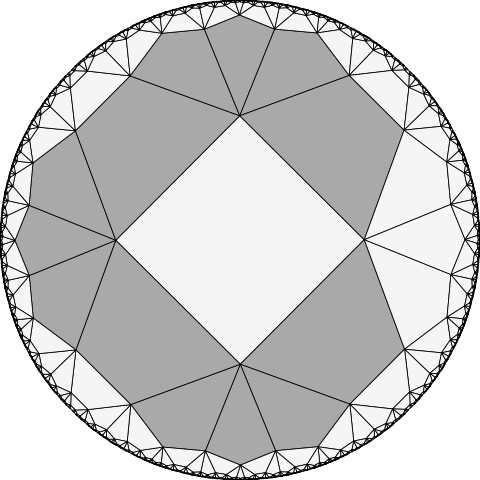}
    \includegraphics[width=0.45\linewidth]{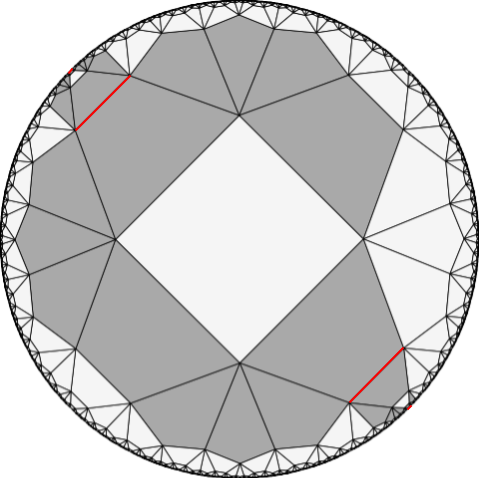}
    \includegraphics[width=0.45\linewidth]{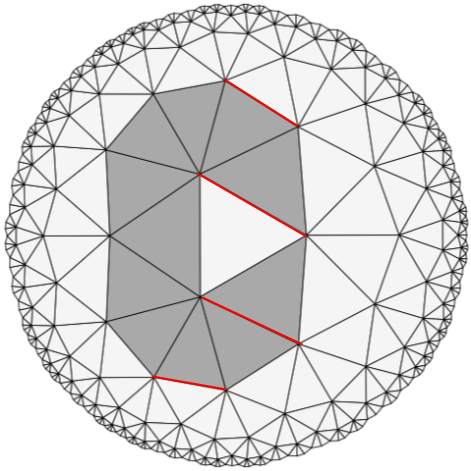}
    \includegraphics[width=0.45\linewidth]{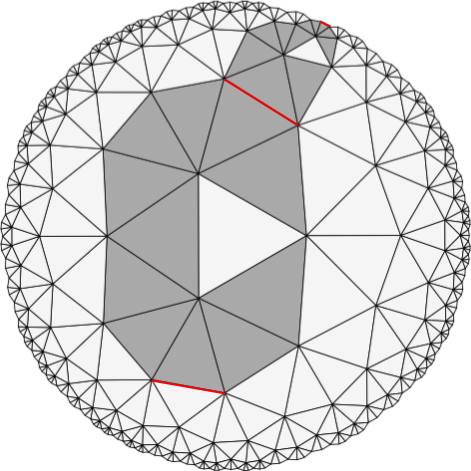}
    \caption{Examples of the polyforms $B_{p,q}(h)$ used in the proof of Theorem~1.6.
The highlighted edges are ultraparallel and serve as reflection axes in the construction
of polyforms with an increasing number of holes.}
    \label{fig:stack}
\end{figure}
\end{proof}
As mentioned in the introduction, these bounds are surprisingly close to each other. 
\begin{cor}
For $p,q>3$,
\begin{align}
    (p-1)(q-3) -1 \leq \frac{g(h)}{h} \leq (p-1)(q-2) + 1.
\end{align}    
\end{cor}
\begin{proof}
    It follows directly from Theorem \ref{thm:bounds} and the fact that $\lfloor \beta \rfloor =q-3$ . See Lemma 14 in \cite{isoperimetric}.
 \end{proof}

\section{The case $p$ even and $q=3$}  \label{section:q3}

In this section, we conjecture a formula for $g(h)$ when $p=2k$ and $q=3$. We provide examples of $\{2k,3\}$-polyforms with $h$ holes for $h\ge 1$. Based on the construction, these polyforms are expected to be crystallized; however a proof has so far eluded us.  
 
Let $k> 3$ be an integer and let $\vmin^{2k,k}(n)$ be the minimum number of vertices for a $\{2k,k\}$-polyform with $n$ tiles. We show that  
\begin{align}
    g_{2k,3}(h) \le \vmin^{2k,k}(h).
\end{align}
This is achieved by constructing a $\{2k,3\}$-polyform by duality using the vertices of an extremal $\{2k,k\}$-polyform. 

We start by realizing the $\{2k,k\}$-tessellation as a subgraph of the $\{3,2k\}$-tessellation as follows: For each tile in the $\{2k,k\}$-tessellation, we add its center as a vertex and add edges between this center and every vertex of the tile. This creates a regular $\{3,2k\}$-tessellation since all triangles have interior angles equal to $\pi/k$. The vertices of the $\{3,2k\}$-tessellation are colored gray if they were originally in the $\{2k,k\}$-tessellation, and the vertices are colored white if they are added centers. 

Finally, in the $\{2k,3\}$-tessellation, we regard as holes those tiles that correspond, by duality with the $\{3, 2k\}$-tessellation, to the vertices that were the centers of the $\{2k,k\} $-tiles. The final result is called the \textit{punctured} $\{2k,3\}$-tessellation. See Figure \ref{fig:q3-tess}.
The puncturing process can be restricted to a $\{2k,k\}$-polyform with $h$ tiles to create $\{2k,3\}$-polyforms with $h$ holes.

\begin{figure}[htbp] 
    \centering
    \includegraphics[height=180px, width=.4\linewidth]{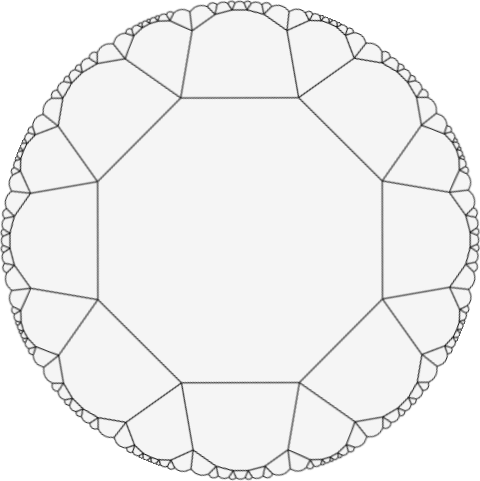} 
    \includegraphics[height=180px, width=.4\linewidth]{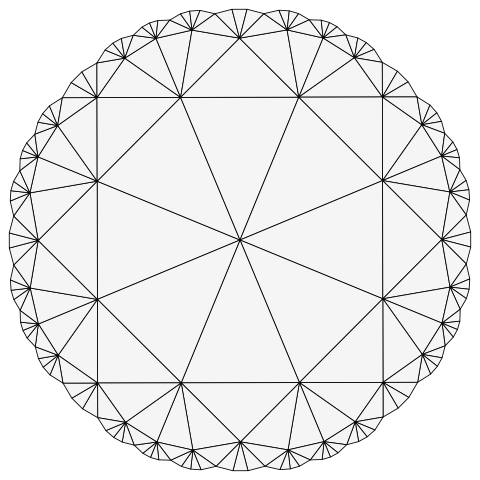}    
    \includegraphics[height=180px, width=.4\linewidth]{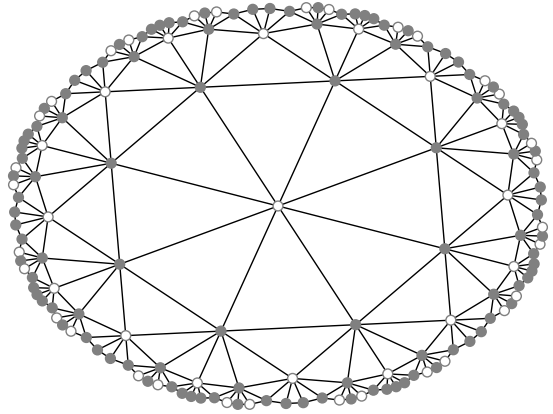} 
    \includegraphics[height=180px, width=.4\linewidth]{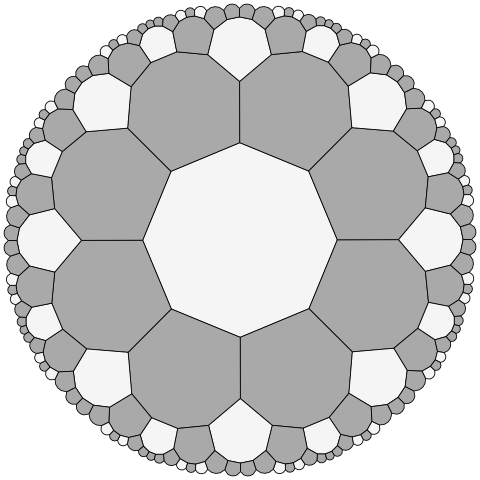}
    \caption{The puncturing process. \textit{Top left.} The $\{2k,k\}$-tessellation.  \textit{Top right.}  For each tile in the $\{2k,k\}$-tessellation, its center is added as a vertex, and edges are added between this center and the vertices of the tile, resulting in the $\{3, 2k\}$-tessellation. \textit{Bottom left.} The vertices of the $\{3, 2k\}$-tessellation are colored gray and white. The subgraph of gray vertices is the original $\{2k, k\}$-tessellation; the white vertices are the added centers. \textit{Bottom right.} The dual of $\{3, 2k\}$ is the $\{2k,3\}$-tessellation with an associated coloring of tiles. }
    \label{fig:q3-tess}
\end{figure}

\begin{defn}
Let $A$ be a $\{2k,k\}$-polyform. Consider it as a subgraph of the $\{3,2k\}$-tessellation as with the puncturing process and let $\tilde{A}$ be the set of tiles in the $\{2k,3\}$-tessellation such that each vertex of $A$ corresponds to a tile in $\tilde{A}$. We call  $\tilde{A}$ the \textit{punctured polyform} with respect to $A$. See Figure \ref{fig:q3}. 
\end{defn}

\begin{figure}[htbp] 
    \centering
    \includegraphics[height=180px, width=.4\linewidth]{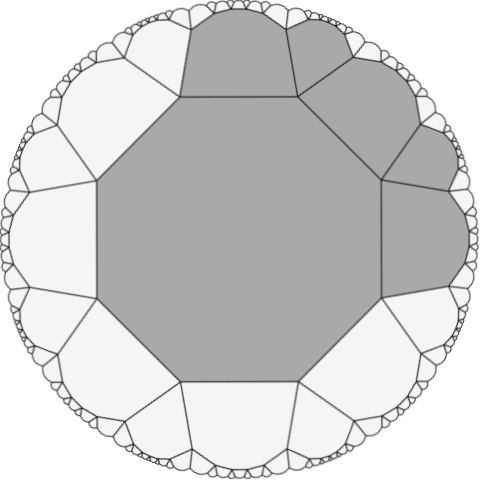} 
    \includegraphics[height=180px, width=.4\linewidth]{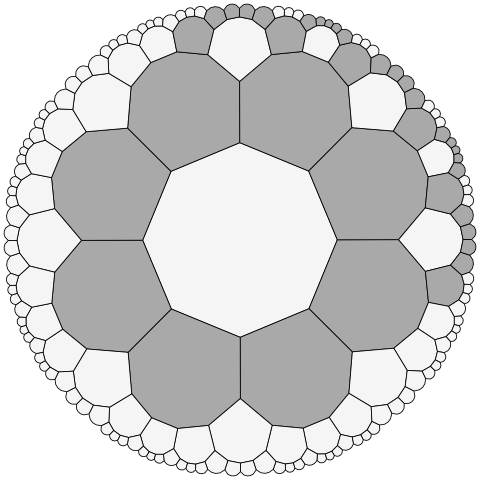} 
    \caption{\textit{Left.} An extremal polyform with 6 tiles in the $\{2k,k\}$-tessellation when $k=4$. \textit{Right.} The corresponding punctured $\{2k,3\}$-polyform with 6 holes.}
    \label{fig:q3}
\end{figure}

\begin{prop} \label{prop:Atilde}
Let $A$ be a $\{2k,k\}$-polyform with $h$ tiles and $v(A)$ vertices. Then $\tilde{A}$ is a $\{2k,3\}$-polyform with $h$ holes and $v(A)$ tiles.
\end{prop}
\begin{proof}
Note that placing a vertex in the center of every face of the $\{2k,k\}$-tessellation and adding an edge between the center and exterior vertices results in a $\{3,2k\}$-tessellation because one additional edge is added between each edge of the original tessellation. When we take the dual, this becomes a $\{3,2k\}$-tessellation.

Every tile $X$ in $\tilde{A}$ corresponds to a vertex $X'$ of a tile in $A$. The incidence relations of the tiles in $\tilde{A}$ correspond to the edge connections of the vertices in $A$. The vertices of a tile are connected, so if two tiles $X_1,X_2$ have corresponding $X_1',X_2'$ on the same tile in $A$, they are connected. If two tiles are incident in $A$, they must share some vertices. Connectedness is a transitive property; so if two tiles share vertices, every vertex in both tiles is connected to every other vertex. Because $A$ is connected, $\tilde{A}$ must also be connected.

Now, we focus on the graph parameters of $\tilde{A}$. The center of each tile in $A$ corresponds to a hole in $\tilde{A}$; therefore, $h(\tilde{A})=|A|=h$. Finally, each vertex of $A$ corresponds to a tile in $\tilde{A}$, that is, $\tilde{A}$ has $v(A)$ tiles.
\end{proof}

We are now ready to provide an upper bound for $g_{2k,3}(h)$ using the construction of $\tilde{A}$ when $A$ is an extremal polyform, thus proving Theorem \ref{thm:q3} which we restate here for the reader's convenience. 

\begin{thm*}[Theorem \ref{thm:q3}]
For $k>3$ an integer, 
    \begin{align}
        g_{2k,3}(h) \leq  \frac{\Pmin^{2k,k}(h)+(2k-2)h}{2}+1.
    \end{align}  
\end{thm*}
\begin{proof} 
Let $A$ be a $\{p,q\}$-polyform with $h$ tiles (and no holes), $v$ vertices, $e$ edges, and $p_o$ perimeter edges. By Euler's formula, we have $v-e+h=1$. Counting the edges of all tiles in $A$, allowing for duplication of shared edges, gives $2e-p_o = p\cdot h$. Combining these two relations yields  $v=\frac{p_o+(p-2)h}{2}+1$. This shows that, for a fixed number of tiles $h$, minimizing $p_o$ is equivalent to minimizing $v$. 

Recall that a polyform $A$ is extremal if its perimeter $p_o$ is the smallest among all polyforms with $h$ tiles. By the previous observation, $A$ also minimizes $v$. This justifies the use of the following notation for extremal polyforms with $h$ tiles: $\vmin^{p,q}(h)=v$ and $\Pmin^{p,q}(h)=p_o$. The equation $v=\frac{p_o+(p-2)h}{2}+1$ then translates to:
    \begin{align}
   \vmin^{p,q}(h) = \frac{\Pmin^{p,q}(h)+(p-2)h}{2}+1.     
    \end{align}    

Let $A$ be an extremal $\{2k,k\}$-polyform with $h$ tiles; for example, the spiral polyform from \cite{MaRoTo}; see Definition 4.1 and Theorem 1.8 therein. Using Proposition \ref{prop:Atilde}, we can construct $\tilde{A}$, a $\{2k,3\}$-polyform with $\vmin^{2k,k}(h)$ tiles and $h$ holes. This construction proves:
    \begin{align*}
    g_{2k,3}(h) \le \vmin^{2k,k}(h)    
    \end{align*}
It immediately follows that:
\begin{align*}
    g_{2k,3}(h) \le \frac{\Pmin^{2k,k}(h)+(2k-2)h}{2}+1.
\end{align*}
\end{proof}

\begin{conj} \label{conj:q3}
If $A$ is an extremal $\{2k,k\}$-polyform, then $\tilde{A}$ is crystallized. That is, 
\begin{align} \label{eqn:g3}
    g_{2k,3}(h) = \frac{\Pmin^{2k,k}(h)+(2k-2)h}{2}+1.
\end{align}
\end{conj}
If this conjecture holds, then we get the asymptotic behavior of $g_{2k,3}$. Recall the asymptotic behavior of $\Pmin$:
\begin{align*}
    \lim_{n\to \infty} \frac{\Pmin^{p,q}(n)}{n} = p-2-\frac{2}{\beta}.
\end{align*}
For the case $p=2k$ and $q=k$, we have:
\begin{align*}   
    \beta & =\frac{(2k-2)(k-2)+\sqrt{(2k-2)^2(k-2)^2-4(2k-2)(k-2)}}{2(2k-2)}.
\end{align*} 
Then, 
\begin{align*}   
      \frac{1}{\beta}  &= \frac{2}{k-2+\sqrt{(k-2)^2-4\frac{k-2}{2k-2}}}, \\
                  &= \frac{2\left(k-2 - \sqrt{(k-2)^2-4\frac{k-2}{2k-2}}\right)}{4\frac{k-2}{2k-2}}, \\
        &= k-1 -\sqrt{(k-1)^2-\frac{2k-2}{k-2}}.      
\end{align*} 
Substituting into \ref{eqn:g3} yields:
\begin{align*}
    \lim_{n\to \infty} \frac{g_{2k,3}(h)}{h} &= \frac{2k-2-2/\beta+2k-2}{2}, \\
                    &= 2k-2-\frac{1}{\beta}, \\
                    &=  k-1+ \sqrt{(k-1)^2-\frac{2k-2}{k-2}}.      
    \end{align*}

\section{Conclusions and further directions}

A key ingredient in our approach is the quantity $M(n,h)$, which yields an algebraic obstruction to the existence of $\{p,q\}$--polyforms with $n$ tiles and $h$ holes: whenever $h > M(n,h)$, no such polyform can exist. This obstruction is purely analytic and depends only on the isoperimetric profile $\mathcal{P}^{p,q}_{\min}$.

Evaluating the inequality $h \le M(n,h)$ at successive values of $n$ provides a natural hierarchy of obstructions. In particular, computing $M(n-k,h)$ for small integers $k\ge 1$ allows one to rule out entire ranges of tile counts. For example, the condition $M(n-2,h) < h$ rigorously excludes the existence of $\{p,q\}$--polyforms with $n-2$ tiles and $h$ holes, while $M(n-1,h) \ge h$ indicates that the algebraic obstruction alone does not rule out the case $n-1$. This phenomenon creates a narrow ``ambiguity window'' in which minimality cannot be decided by the $M$--bound alone.

The tables in Section~\ref{section:bounds} illustrate how this multi--step obstruction can be used in practice to certify minimality, to identify crystallized cases, and to isolate parameter regimes where additional geometric or combinatorial arguments are required. More generally, this suggests the following directions for future work.

\begin{itemize}
    \item Determine for which hyperbolic tessellations $\{p,q\}$ the algebraic threshold
    \[
        n_{\mathrm{alg}}(h) := \min\{n \in \mathbb{N} : h \le M(n,h)\}
    \]
    coincides with the true minimum $g_{p,q}(h)$ for all sufficiently large $h$.

    \item Investigate whether the difference $g_{p,q}(h) - n_{\mathrm{alg}}(h)$ remains uniformly bounded, or whether it exhibits structured behaviour depending on $p$, $q$, or congruence classes of $h$.

    \item Develop geometric or constructive criteria that resolve the ambiguity windows detected by inequalities of the form
    \[
        M(n-k,h) < h \le M(n-k+1,h),
    \]
    thereby bridging the gap between algebraic obstructions and explicit realizations.
\end{itemize}

Taken together, these questions suggest that the quantities $M(n-k,h)$ provide a flexible framework for systematically excluding large classes of parameter tuples $(p,q,n,h)$, and for guiding the search for new extremal or near--extremal hyperbolic polyforms.

The bounds on $g(h)$ given in Theorem \ref{thm:bounds} are not sharp. Our construction of holey polyforms given in Section \ref{section:bounds} was  simple due to the technicalities arising in hyperbolic geometry; hence it is not optimal. A general procedure (that works for all $\{p,q\}$) seems difficult to obtain, as we have not recognized any patterns. A more systematic study of examples is needed to understand the underlying structure of holes in crystallized polyforms. For example, for $p=3$, the construction of $Spir_k$ given in \cite{MaR} may be replicated in the hyperbolic context, it is not clear exactly how to do this, and if it will result in crystallized polyforms. Another approach is to follow the dual construction seen in Section \ref{section:q3} for a general $\{p,q\}$, but we may need to rely on non-regular tessellations of the hyperbolic plane and on a notion of duality for holey polyforms that captures the structure of the holes. 

As noted in Section \ref{section:table}, our search algorithm is not exhaustive. Future work will involve developing algorithmic enumeration methods in hyperbolic settings. For instance, the algorithm for generating geodesic regular tree structures (GRTS) presented in \cite{celinska2021generating} could inspire exhaustive-generation strategies, particularly for isomorphism rejection via dual-graph canonical forms.

As mentioned in conjecture \ref{conj:q3}, we believe that the bound for $g(h)$ in the $\{2k,3\}$-case is sharp. A proof will require either a characterization of crystallized polyforms (e.g. holes must have size 1) or a better understanding of the parameters of holey polyforms, for example, improved bounds on their number of vertices.

\subsection*{Acknowledgements}
We would like to thank Peter Kagey for helpful discussions and the Mason Experimental Geometry Lab at George Mason University for facilitating this project.

\subsection*{Competing Interest Statement}
The authors have no relevant financial or non-financial interests to disclose. The authors have no conflicts of interest to declare relevant to this article's content. All authors certify that they have no affiliations with or involvement in any organization or entity with any financial or non-financial interest in the subject matter or materials discussed in this manuscript. The authors have no financial or proprietary interests in any material discussed in this article.

\appendix
\section{Notation} \label{appendix}
This section summarizes the notation used in the paper. \\

Parameters of a $\{p,q\}$-polyform $A$:
\begin{itemize}
    \item $|A|, v(A), e(A), h(A)$ are the number of tiles, vertices, edges, and holes, respectively.
    \item $p_o(A), p_h(A), b(A)$ are the outer perimeter, hole perimeter, and interior edges respectively.
    \item $H(A)$ is the union of the holes in $A$. 
    \item $\overline{A}=A\cup H(A)$ is the polyform with holes filled in.    
    \item $A'$ is the dual graph of $A$. The vertices of $A'$ are tiles of $A$. An edge in $A'$ is drawn whenever the corresponding tiles in $A$ share an edge. 
\end{itemize}

Relations between graph parameters of $A$:
\begin{itemize}
    \item Euler's Formula for $A$: $v(A)-e(A) + |A|+h(A)=1$.
    \item Number of edges of $A$: $e(A)=p_o(A)+b(A)+p_h(A)$.
    \item Counting the edges of all tiles in $A$, allowing for duplication of shared edges: $p\cdot |A| = p_o(A)+ 2 b(A)+p_h(A) = e(A)+b(A) = 2e(A)-b(A)$.
\end{itemize}

Conventions for variables: 
\begin{itemize}
    \item $k$ is a positive integer indexing rings of tiles in a $\{p,q\}$-tessellation.
    \item $n$ is a positive integer referring to the number of tiles of a polyform.
    \item $\{p,q\}$ is the Schlaffli symbol of a tessellation.
    \item $\beta$ is a parameter of a $\{p,q\}$-tessellation, see Equation \ref{eqn:beta}. It appears in the asymptotic behaviour of $\Pmin^{p,q}(n)$.
    \item $A_{p,q}(k)$ is the complete $k$-layered $\{p,q\}$-polyform. Defined as follows: Let $A_{p,q}(1)$ denote the $\{p,q\}$-polyform with only one $p$-gon. Then, construct $A_{p,q}(k)$ from $A_{p,q}(k-1)$ by adding precisely the tiles needed so that all the perimeter vertices of $A_{p,q}(k-1)$ are surrounded by $q$ tiles. We call $A_{p,q}(k)$ the \emph{complete $k$-layered $\{p,q\}$-polyform}. See Definition 1.2 in \cite{MaRoTo} 
    \item $S_{p,q}(n)$ is the spiral polyform with $n$ tiles. They serve as benchmarks for the extremal values attained by polyforms, e.g., they minimize outer perimeter. See Definition 4.1 and Theorem 1.8 in \cite{MaRoTo}.
    \item $A,B,C$ are generic polyforms.
    \item $l, m$ are hyperbolic lines.
    \item  $X,Y,Z$ are tiles  
    \item $x,y,z$ are vertices 
\end{itemize}

\bibliographystyle{abbrvnat}
\bibliography{polyforms}

\end{document}